\documentclass[11pt]{amsart}

\usepackage{subfiles}

\usepackage{tabu}
\usepackage[margin=1in]{geometry}
\usepackage[dvipsnames]{xcolor}   		             		
\usepackage{graphicx}			
\usepackage{amssymb}
\usepackage{mathrsfs}
\usepackage{amsthm}
\usepackage{amsmath}
\usepackage{stmaryrd}
\usepackage{tikz}
\usepackage{tikz-cd}
\usepackage{accents}
\usepackage{upgreek}
\usepackage{enumerate}
\usepackage{bm}
\usepackage{mathtools}
\usepackage[all]{xy}
\usepackage{caption}

\usepackage{url}
\usepackage{float}
\usepackage{todonotes}
\usepackage{bbm}
\usepackage{enumitem}

\usepackage[full]{textcomp}
\usepackage[cal=cm]{mathalfa}
\usepackage{xparse}
\usepackage{comment}
\usepackage{cite}
\usepackage{xfrac}
\usepackage[normalem]{ulem}

\tikzset{
  commutative diagrams/.cd, 
  arrow style=tikz, 
  diagrams={>=stealth}
}

\theoremstyle{definition}

\makeatletter
\def\@tocline#1#2#3#4#5#6#7{\relax
  \ifnum #1>\c@tocdepth 
  \else
    \par \addpenalty\@secpenalty\addvspace{#2}%
    \begingroup \hyphenpenalty\@M
    \@ifempty{#4}{%
      \@tempdima\csname r@tocindent\number#1\endcsname\relax
    }{%
      \@tempdima#4\relax
    }%
    \parindent\z@ \leftskip#3\relax \advance\leftskip\@tempdima\relax
    \rightskip\@pnumwidth plus4em \parfillskip-\@pnumwidth
    #5\leavevmode\hskip-\@tempdima
      \ifcase #1
       \or\or \hskip 1em \or \hskip 2em \else \hskip 3em \fi%
      #6\nobreak\relax
    \dotfill\hbox to\@pnumwidth{\@tocpagenum{#7}}\par
    \nobreak
    \endgroup
  \fi}
\makeatother

\usetikzlibrary{calc}
\usetikzlibrary{fadings}
\usetikzlibrary{decorations.pathmorphing}
\usetikzlibrary{decorations.pathreplacing}

\newcounter{marginnote}
\DeclareMathAlphabet{\mathpzc}{OT1}{pzc}{m}{it}

\usepackage[backref]{hyperref}
\hypersetup{
  colorlinks   = true,          
  urlcolor     = blue,          
  linkcolor    = purple,          
  citecolor   = blue             
}

\theoremstyle{definition}
\newtheorem{theorem}{Theorem}[section]

\newtheorem{corollary}[theorem]{Corollary}
\newtheorem{lemma}[theorem]{Lemma}
\newtheorem{proposition}[theorem]{Proposition}
\newtheorem{remark}[theorem]{Remark}

\newtheorem*{runningexample*}{Running example}

\newtheorem*{aside*}{Aside}

\newtheorem{definition}[theorem]{Definition}
\newtheorem{example}[theorem]{Example}

\newtheorem{proposition-definition}[theorem]{Proposition-Definition}

\DeclareMathOperator{\Hom}{Hom}

\DeclareMathOperator{\Pic}{Pic}

\newcommand{\bcd}{\begin{center}\begin{tikzcd}}
\newcommand{\ecd}{\end{tikzcd}\end{center}}

\newcommand{\cC}{\mathcal{C}}
\newcommand{\cE}{\mathcal{E}}
\newcommand{\cF}{\mathcal{F}}

\newcommand{\cL}{\mathcal{L}}
\newcommand{\cO}{\mathcal{O}}
\newcommand{\cQ}{\mathcal{Q}}

\newcommand{\cM}{\mathcal{M}}

\newcommand{\cX}{\mathcal{X}}
\newcommand{\cT}{\mathcal{T}}
\newcommand{\cY}{\mathcal{Y}}

\newcommand{\qlogo}{\cQ_{0, \alpha}^{\mathrm{log}}(X|D,\beta)}

\newcommand{\qlog}{\cQ_{g, \alpha}^{\mathrm{log}}(X|D,\beta)}

\newcommand{\largmr}{\frM_{g,\alpha}^{\mathrm{log}}(\argmr)}
\newcommand{\qrel}{\mathcal{Q}^{\mathrm{rel}}_{0,\alpha}(X|D,\beta)}

\newcommand{\CC}{\mathbb{C}}

\newcommand{\LL}{\mathbb{L}}
\newcommand{\NN}{\mathbb{N}}

\newcommand{\PP}{\mathbb{P}}
\newcommand{\QQ}{\mathbb{Q}}
\newcommand{\RR}{\mathbb{R}}
\newcommand{\TT}{\mathbb{T}}
\newcommand{\ZZ}{\mathbb{Z}}
\newcommand{\AAA}{\mathbb{A}}
\newcommand{\GG}{\mathbb{G}}
\newcommand{\agm}{[\AAA^1/\GG_{m}]}
\newcommand{\argmr}{[\AAA^r/\GG_{m}^r]}

\newcommand{\vir}{\text{\rm vir}}

\newcommand{\frM}{\mathfrak{M}}

\newcommand{\Spec}{\operatorname{Spec}}

\newcommand{\GIT}{/\!\!/}

\NewDocumentCommand{\compatibilitydatum}{m m m m m m O{} O{} O{}}{
\begin{equation*} \begin{tikzcd}[ampersand replacement=\&]
  \: \arrow{r} \& {#1} \arrow{r} \arrow{d}{#7} \& {#2} \arrow{r} \arrow{d}{#8} \& {#3} \arrow{r}{[1]} \arrow{d}{#9} \& \: \\
  \: \arrow{r} \& {#4} \arrow{r} \& {#5} \arrow{r} \& {#6} \arrow{r} \& \:
\end{tikzcd} \end{equation*}}

\NewDocumentCommand{\commutingsquare}{m m m m o O{} O{} O{} O{}}{
\begin{equation}\begin{tikzcd}[ampersand replacement=\&] \label{#5}
  #1 \arrow{r}{#6} \arrow{d}{#7} \& #2 \arrow{d}{#8} \\
  #3 \arrow{r}{#9} \& #4
\end{tikzcd}\IfValueTF{#5}{\label{#5}}{} \end{equation}}

\NewDocumentCommand{\cartesiansquare}{m m m m O{} O{} O{} O{}}{
\begin{equation*}\begin{tikzcd}[ampersand replacement=\&]
  #1 \arrow{r}{#5} \arrow{d}{#6} \arrow[dr, phantom, "\square"] \& #2 \arrow{d}{#7} \\
  #3 \arrow{r}{#8} \& #4
\end{tikzcd} \end{equation*}}

\NewDocumentCommand{\cartesiansquarelabel}{m m m m m O{} O{} O{} O{}}{
\begin{tikzcd}[ampersand replacement=\&]
  #1 \arrow{r}{#6} \arrow{d}{#7} \arrow[dr, phantom, "\square"] \& #2 \arrow{d}{#8} \\
  #3 \arrow{r}{#9} \& #4
\end{tikzcd}\IfValueTF{#5}{\label{#5}}{}
}

\NewDocumentCommand{\triangleofspaces}{m m m O{} O{} O{}}{
\begin{tikzcd} [ampersand replacement=\&]
#1 \arrow{r}{#4} \arrow[bend right]{rr}{#5} \& #2 \arrow{r}{#6} \& #3
\end{tikzcd}}

\begin{document}
 
\title{Logarithmic Quasimaps}

\author[Shafi]{Qaasim Shafi}
\address{School of Mathematics\\
	Watson Building \\
	University of Birmingham\\
        Edgbaston
	B15 2TT\\
	UK}
\email{m.q.shafi@bham.ac.uk}

\begin{abstract}
	We construct a proper moduli space which is a Deligne--Mumford stack parametrising quasimaps relative to a simple normal crossings divisor in any genus using logarithmic geometry. We show this moduli space admits a virtual fundamental class of the expected dimension leading to numerical invariants which agree with the theory of Battistella--Nabijou where the latter is defined.
\end{abstract}

\maketitle
\vspace{-3em}
\tableofcontents

\vspace{-3em}

\section*{Introduction}

\subsection{Results}

Let $X$ be a GIT quotient $W \GIT G$ of an affine variety $W$ by a reductive algebraic group $G$, satisfying the assumptions of \cite{ciocan2014stable}. Examples include toric varieties, partial flag varieties and complete intersections in these spaces. In addition, suppose $X$ is a subvariety of $V \GIT G$, where $V$ is a vector space, and let $D = D_1 + \dots + D_r$ be a simple normal crossings divisor on $X$ pulled back from $V \GIT G$. For many targets, for example toric varieties and  partial flag varieties, $V=W$ and hence this covers all simple normal crossings divisors. We build the theory of logarithmic quasimaps for $(X,D)$.

\begin{theorem}
	Fix non-negative integers $g,n$, an effective degree $\beta$ and $\alpha = (\alpha_{i,j})_{i,j}$ a matrix of non-negative integers with $\sum_{j=1}^n \alpha_{i,j} = D_{i} \cdot \beta$. The moduli space $\qlog$ parametrising logarithmic quasimaps to $(X,D)$ of genus $g$, degree $\beta$ with contact orders $\alpha$ is a proper Deligne--Mumford stack.
\end{theorem}

Given a quasimap to $X$ from a curve $C$, each $D_i$ induces a line bundle-section pair on $C$, thought of as the pullback of the pair cutting out $D_i$. The moduli space parametrises quasimaps to $X$ together with a logarithmic enhancement (with contact order $\alpha$) of the morphism $C \rightarrow [\AAA^r/\GG_m^r]$ induced by these line bundle-section pairs. This compactifies the space of maps from smooth curves to $X$ with contact order $\alpha$ along $D$. As is standard, the moduli space is not smooth admitting a fundamental class, but we show it admits a virtual fundamental class.

\begin{theorem}
	The moduli space $\qlog$ admits a perfect obstruction theory over $\frM_{g,\alpha}^{\log}([\AAA^r/\GG_m^r])$ leading to a virtual fundamental class $[\qlog]^{\vir}$.
\end{theorem} 

In \cite{battistella2021relative}, the authors build a theory of relative quasimaps for a smooth projective toric variety relative to a smooth, very ample divisor in genus zero. 

\begin{theorem}
	For $X$ a smooth projective toric variety, $g=0$ and $D$ smooth and very ample, this theory coincides with the theory of Battistella--Nabijou. 
\end{theorem}

\subsection{Why Quasimaps?}

Relative or logarithmic Gromov--Witten theory has had a tremendous influence on enumerative geometry in recent years. It has proved important for modern constructions in mirror symmetry \cite{gross2022canonical}, for determining ordinary Gromov--Witten invariants via the degeneration formula \cite{li2002degeneration, ranganathan2019logarithmic, chen2014degeneration, abramovich2016orbifold, kim2021degeneration, abramovich2020decomposition}, and for providing insights about the moduli space of curves \cite{graber2005relative}.  Quasimap theory provides an alternative curve counting framework \cite{marian2011moduli,ciocan2010moduli, ciocan2014stable} when the target admits a certain GIT presentation, by incorporating basepoints. The resulting \emph{quasimap invariants}, have also proved important in the context of mirror symmetry, as well as for studying the moduli space of curves. One example is their use in determining relations in the $\kappa$ ring \cite{pandharipande2012kappa}. Most crucially though, there are wall-crossing formulas relating Gromov--Witten invariants and quasimap invariants \cite{ciocan2014wall, ciocan2020quasimap, zhou2022quasimap}. We expect that a theory of logarithmic quasimaps will be able to produce new insights in logarithmic Gromov--Witten theory and its neighbouring areas.

\subsection*{Logarithmic Wall-Crossing} Computations in logarithmic Gromov--Witten theory are well sought after. The case of a smooth pair is well understood, yet explicit computations in the simple normal crossings case have proved to be more elusive. The few tools available use tropical geometry \cite{nishinou2006toric, mandel2020descendant, ranganathan2017skeletons}, scattering diagrams \cite{gross2010tropical, bousseau2020quantum, graefnitz2022tropical} and, more recently, rank reduction \cite{nabijou2022gromov}.
For ordinary (non-logarithmic) Gromov--Witten theory, almost all results that allow us to compute genus-zero Gromov--Witten invariants rely at heart on the wall-crossing formula \cite{ciocan2014wall,ciocan2020quasimap, zhou2022quasimap}, the comparison between quasimap invariants and Gromov--Witten invariants.  Battistella and Nabijou have proposed a similar program of computing logarithmic Gromov--Witten invariants by proving a wall-crossing formula relating logarithmic quasimaps and logarithmic Gromov--Witten invariants \cite{battistella2021relative}. Furthermore, they found as evidence for their proposal that a certain generating function of genus zero quasimap invariants relative to a smooth divisor coincided with the relative $I$-function of \cite{fan2019mirror}, which in turn can be obtained from a generating function of relative Gromov--Witten invariants by a change of variables. A full theory of logarithmic quasimaps now allows this avenue to be pursued.

\subsection*{Mirror Symmetry} Quasimaps have a fundamental connection to mirror symmetry. The wall-crossing formula is exactly the mirror map for Calabi-Yau threefolds \cite{ciocan2014wall}, and so the quasimap invariants coincide on the nose with B-model invariants of the mirror. The (conjectural) holomorphic anomaly equation provides remarkable structure to the Gromov--Witten theory of a Calabi-Yau, coming from the B-model. The link between the quasimap and B-model invariants has been utilised to give a direct geometric proof of the holomorphic anomaly equation for local $\PP^2$ \cite{lho2018stable}. In another direction, the authors of \cite{bousseau2021holomorphic} prove a holomorphic anomaly equation for the logarithmic Gromov--Witten theory of $\PP^2$ relative an elliptic curve. This provides interesting directions for holomorphic anomaly equations for logarithmic quasimaps.

\subsection*{Logarithmic-Orbifold Comparison} There is another approach to counting curves with tangency conditions, using orbifolds \cite{cadman2007using, tseng2020gromovwitten}. Here, one takes a pair $(X,D_1 + \dots + D_r)$ and replaces the divisor components with roots of the divisor, introducing non-trivial isotropy groups. In this orbifold compactification the tangency gets recorded in the group homomorphism between isotropy groups, using the technology of orbifold stable maps. If the divisor is smooth, then these two approaches give the same invariants in genus zero \cite{abramovich2017relative}. Roughly speaking, this equivalence can be seen by taking the coarse moduli map of an orbifold stable map, and noting that there is a logarithmic lift. 

This is no longer true in the simple normal crossings case \cite{nabijou2022gromov, battistella2022local}. The orbifold theory is more computable since it satisfies a product formula, but gives the wrong answers from the logarithmic perspective. In \cite{battistella2022gromov}, noting that the logarithmic theory is invariant under logarithmic modifications \cite{abramovich2018birational} but that the orbifold theory is not, the authors show that after a certain blow-up the two theories coincide. A related approach is suggested in \cite{you2022relative}, where the indication is that structures in orbifold Gromov--Witten theory will be preserved under modification. These approaches can be summarised as fixing the moduli problem (Gromov--Witten theory) and altering the target. A complimentary strategy is to fix the target and alter the moduli problem to a situation where the orbifold and logarithmic theories coincide. In ~\cite[Remark 5.4]{nabijou2022gromov}, the authors observe in examples, that the error terms occur in the presence of components of the moduli space consisting of stable maps with rational tails. Quasimap theory is designed to remove rational tails. If we consider quasimap theory as our moduli problem instead it is likely that there will be significantly fewer correction terms between the logarithmic and orbifold theories, the latter of which was developed in \cite{cheong2015orbifold}.

At least in genus zero, the proposed picture is
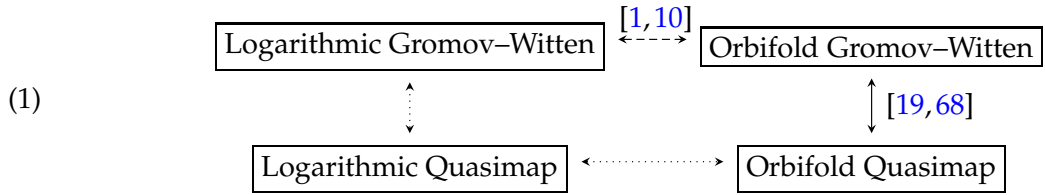
\begin{equation}\label{diagram}
	\begin{tikzcd}
	{\fbox{\txt{Logarithmic Gromov--Witten}}} \arrow[d, leftrightarrow, dotted] \arrow[r, "\txt{\cite{battistella2022gromov, abramovich2017relative}}",leftrightarrow,dashed] & {\fbox{\txt{Orbifold Gromov--Witten}}} \arrow[d,leftrightarrow, "\txt{\cite{zhou2022quasimap, cheong2015orbifold}}"] \\
	\fbox{{\txt{Logarithmic Quasimap}}} \arrow[r,leftrightarrow, dotted]                 & {\fbox{\txt{Orbifold Quasimap}}}              
	\end{tikzcd}
\end{equation}
This paper constructs the bottom left-hand corner of the diagram. The expected wall-crossing formula relating logarithmic Gromov--Witten theory to logarithmic quasimap theory would be the left-hand vertical arrow; the bottom arrow would be the comparison between logarithmic and orbifold quasimap theory. We expect this last comparison to be simpler, and in certain situations, trivial, which we will show a particular instance of, see Example \ref{p22linesQvsGW}.

 A parallel direction is the local-logarithmic correspondence \cite{vangarrel2019local}, which exhibits some of the beautiful geometry in logarithmic Gromov--Witten theory of a smooth pair $(X,D)$. There are counterexamples to a natural generalisation when $D$ is a simple normal crossings divisor, which are the same counterexamples used in the logarithmic-orbifold comparison \cite{nabijou2022gromov}, since the local and orbifold theories coincide \cite{battistella2022local}. 

 Calculations in \cite{shafi2022enumerative} indicate that these are no longer counterexamples in the quasimap setting. Therefore in an analogue comparison to the bottom arrow of Diagram \eqref{diagram}, it is likely that the local-logarithmic correspondence will hold in greater generality in the quasimap setting. 

\subsection{Outline}

We begin in Section \ref{2} by recalling the definition of quasimaps and presenting a way to incorporate divisors using maps to $[\AAA^r/\GG_m^r]$. 

In Section \ref{modq} we build the moduli space of logarithmic quasimaps and show it is a proper, Deligne--Mumford stack. The difficulty in the definition is that a priori it was not clear how the logarithmic structure should interact with the basepoints. The key is to recognise that in the Gromov--Witten setting a (stable) logarithmic map can be decomposed into the underlying stable map together with a logarithmic morphism to the Artin fan (or $[\AAA^r/\GG_m^r]$). We use this decomposition as an analogy for defining a logarithmic quasimap. Moreover, this approach allows for a streamlined presentation of the theory, which relies on the existence and properties of the space of logarithmic maps to $[\AAA^r/\GG_m^r]$. 

In Section \ref{4} we present three examples of the moduli space for $\PP^N$ relative to a collection of hyperplanes and compare them to the Gromov--Witten setting. For a single hyperplane in genus zero, we note that both spaces are irreducible but that the boundary compactification is significantly simpler in the quasimap case. In the case of multiple hyperplanes in genus zero, we give a specific instance where wall-crossing accounts for the entire discrepancy between orbifold Gromov--Witten and logarithmic Gromov--Witten theory. The final example exhibits the fact that with respect to the full toric boundary in any genus, the logarithmic Gromov--Witten and quasimap moduli spaces coincide. 

In Section \ref{5} we construct the virtual fundamental class, first in generality and then noting how the construction simplifies when working with a smooth projective toric variety relative to a toric divisor. 

Finally in Section \ref{BN}, we prove that this theory of logarithmic quasimaps coincides with the Battistella--Nabijou theory of relative quasimaps from \cite{battistella2021relative}, where the latter is defined. We do this by proving the result for $\PP^N$ relative to a hyperplane and then using the very ample embedding to pull the result back.

\subsection{History} Over the years there has been much interest in defining and computing relative Gromov--Witten invariants of a pair $(X,D)$. When attempting to define a proper moduli space of curves with fixed contact orders to $D$ one sees that in the limit whole components of the curve can fall into $D$, at which point it is no longer clear what contact order means. This poses a major difficulty. A first solution was proposed by Gathmann \cite{gathmann2002absolute} building on work of Vakil \cite{vakil2000enumerative} for (very) ample smooth divisors in genus zero by taking a closure inside the space of absolute stable maps. Jun Li \cite{li2001stable} defined a theory of relative stable maps for any smooth divisor in all genus using expansions of the target. Since then, there have been various approaches to extend the theory to the simple normal crossings case. The first of these came from Abramovich, Chen, Gross and Siebert \cite{gross2013logarithmic,chen2014stable,abramovich2014stable} using logarithmic structures which provide extra structure for defining contact order even if a component falls into $D$. There have also been approaches which combine expansions with logarithmic structures \cite{kim2010logarithmic,ranganathan2019logarithmic}. Finally, there is the approach using orbifolds \cite{cadman2007using, tseng2020gromovwitten}, which can give genuinely different invariants. In a different direction there have been extensions of these approaches \cite{abramovich2021punctured, fan2020structures, fan2021higher} which allow for negative contact orders with the divisor.  There has also been a parallel story from the symplectic perspective of Gromov--Witten theory, which is more closely related to the ideas of Jun Li, see \cite{ionel2003relative, li2001symplectic}.

In the quasimap setting, building on Gathmann's approach, Battistella and Nabijou \cite{battistella2021relative} developed a theory of relative quasimaps in genus zero for smooth projective toric varieties relative a smooth (very) ample divisor by taking a closure inside the absolute quasimap space. This paper provides a quasimap theory relative a s.n.c. divisor $D$ which removes these assumptions.

\subsection*{Acknowledgements}

I would like to wholeheartedly thank Dhruv Ranganathan for suggestions and guidance with this project as well as for invaluable feedback on drafts. I owe a great deal of thanks to Navid Nabijou for many helpful discussions as well as for feedback. I would also like to thank Tom Coates and Rachel Webb for helpful conversations. This work was supported by the EPSRC Centre for Doctoral
Training in Geometry and Number Theory at the Interface, grant number EP/L015234/1 and the UKRI Future Leaders Fellowship through
grant number MR/T01783X/1.

\section{Absolute quasimaps with divisors}\label{2}

Quasimaps are a variation on stable maps, where, roughly, one allows rational maps. In this section we recall the definition of absolute GIT quasimaps from \cite{ciocan2014stable} and explain how we will incorporate divisors.
Let $W= \Spec A$ be an affine algebraic variety with the action by a reductive algebraic group $G$. Let $\theta$ be a character inducing a linearisation for the action. We insist that
\begin{itemize}
	\item $W^s = W^{ss}$
	\item $W^s$ is nonsingular
	\item $G$ acts freely on $W^s$
	\item $W$ has only l.c.i singularities
\end{itemize}
Then quasimaps to $W \GIT G$ are defined as follows.
\begin{definition}\label{absquasimapdef}
Fix non-negative integers $g,n$ and $\beta \in \Hom(\Pic^G W, \ZZ)$. An $n$-marked \textit{stable quasimap} of genus $g$ and degree $\beta$ to $W \GIT G$ is 
	\begin{equation*}
	\left((C, p_1, \dots, p_n), u \right)
	\end{equation*}	
where
	\begin{enumerate}
		\item $(C, p_1, \dots, p_n)$ is an $n$-marked prestable curve of genus $g$.
		\item $u: C \rightarrow [W/G]$ of class $\beta$  (i.e. $\beta(L) = \deg_C u^*L$ for $L \in \Pic^G W$).
	\end{enumerate}
	which satisfy 
	\begin{enumerate}
		\item (Non-degeneracy) there is a finite (possibly empty) set $B \subset C$, disjoint from the nodes and markings, such that $\forall c \in C \setminus B$ we have $u(c) \in W^s$. 
		\item (Stability) $\omega_C(p_1 + \cdots + p_n) \otimes u^*L_{\theta}^{\epsilon}$ is ample $\forall \epsilon \in \QQ_{>0}$, where $L_{\theta}$ is the line bundle on $[W/G]$ associated to the trivial line bundle $\cO_W$ and the character $\theta$.
	\end{enumerate}
\end{definition}

There is a moduli space of stable quasimaps to $X=W \GIT G$ \cite[Definition 4.1.1]{ciocan2014stable}, which we denote $\cQ_{g,n}(X,\beta)$. If $2g-2+n \geq 0$ and $\beta$ is effective with respect to $L_{\theta}$ ~\cite[Definition 3.2.2]{ciocan2014stable}, then this moduli space is a proper, Deligne--Mumford stack admitting a perfect obstruction theory leading to quasimap invariants \cite[Theorem 4.1.2]{ciocan2014stable}. Our notation suppresses the dependence of the moduli space on the GIT presentation of $X$. If $X$ is a smooth projective toric variety, unless specified otherwise, we implicitly assume the presentation comes from the fan of $X$ (see Remark \ref{toric presentation}), agreeing with the theory of toric quasimaps in \cite{ciocan2010moduli}.

Now let $X= W \GIT G \hookrightarrow V \GIT G$ be a subvariety of a vector space quotient, defined with the same linearisation, coming from $W \hookrightarrow V$ (the prequotient embedding always exists ~\cite[Proposition 2.5.2]{ciocan2014stable}), and let $D$ be a \emph{smooth} divisor on $V \GIT G$, defined by a line bundle-section pair $(\cO_{V \GIT G}(D),s_{D})$, which pulls back to give a divisor on $W \GIT G$, which we also denote $D$. In order to build a moduli space of quasimaps to $(X,D)$ we need to make sense of the tangency of a quasimap to $X$ along $D$. For a genuine morphism, tangency is the order of vanishing of the pullback of the section cutting out the divisor. Although a quasimap does not come with a morphism to the target $X$, one can still build an analogous line bundle-section pair.

Since $V$ is a vector space, every line bundle is isomorphic to the trivial line bundle. Any character of $G$ induces a line bundle on $V \GIT G$ and sections of this line bundle correspond to $G$-equivariant sections of the trivial bundle on $V$. Any such pair defines a line bundle and section on the stack $[V/G]$. We further impose that every such line bundle-section pair $V \GIT G$ corresponds uniquely to a line bundle-section pair on the stack $[V/G]$. This imposes a restriction not on the ambient vector space quotient variety, but on the choice of presentation of this ambient variety. This is satisfied, for example, if $V \GIT G$ is a toric variety, with presentation induced by the fan, see Remark \ref{toric presentation}, or if $V \GIT G$ is a partial flag variety,  with GIT presentation given by ~\cite[Section 2.2.2]{coates2022abelian}. 

Given a quasimap $C \rightarrow [W/G]$, we then get a line bundle-section pair $(L_D,u_D)$ on $C$ via pullback from the composition $C \rightarrow [W/G] \rightarrow [V/G]$.

\begin{definition}\label{sminduce}
	Let $(X,D)$ be as above and suppose we have a family of  stable quasimaps over a scheme $S$ from a family of curves $C$ to $X$. Then we define  the morphism \textit{induced by $D$} to be the morphism \begin{equation}\label{sminduceeq}
	    C \rightarrow \agm
	\end{equation}
defined by the above line bundle-section pair on $C$.
\end{definition}

\begin{remark}\label{argr2lbs}
	If instead of a smooth divisor we have a simple normal crossings divisor $D= D_1 + \dots + D_r$, then we get $r$ induced maps to $\agm$, one for each component of $D$. Together these give a morphism 
 \begin{equation}\label{sncinduceeq}
     C \rightarrow [\AAA^r/\GG_m^r]
 \end{equation}
 which we once again call the morphism \emph{induced by $D$}.
\end{remark}

\begin{remark}\label{toric presentation}
    If $X$ is a smooth projective toric variety coming from a fan,  then there is an induced presentation, $X=\AAA^M \GIT\GG_m^s$, ~\cite[Theorem 2.1]{cox1995homogeneous}. In this case $\Pic(X) \cong \Pic([\AAA^M/\GG_m^s]) \cong \ZZ^s$ (see Section \ref{Section : Toric Divisiors}), and so every line bundle-section pair corresponds uniquely to a line bundle-section pair on the stack $[\AAA^M/\GG_m^s]$. 
\end{remark}

Below we include an example where the Picard groups of $V \GIT G$ and $[V/G]$ differ.

\begin{example}
    Consider $\PP^1 = \AAA^3 \GIT \GG_m^2$, where the action is given by $(\lambda,\mu) \cdot (x_0,x_1,t) = (\lambda x_0,\lambda x_1,\mu t)$ and the linearisation is given by the character $(1,1)$. Here, the presentation is not coming from the fan of $\PP^1$. If $D = Z(x_0) \subset \PP^1$ then one can, for example, choose $x_0 \cdot t^{n}$ as the section of the line bundle corresponding to the character $(1,n)$ on $[\AAA^3 \GIT \GG_m^2]$ for any $n \in \NN$. This is reflective of the fact that $\Pic(\PP^1) = \ZZ$ but $\Pic([\AAA^3/\GG_m^2]) = \ZZ^2$.
\end{example}

One can build a theory of logarithmic quasimaps in cases like the above example, but there is an additional choice of a line bundle-section pair on the ambient quotient stack. A canonical choice would just be the closure of $D$, which in the above example corresponds to $n=0$.

\section{Moduli space of logarithmic quasimaps}\label{modq}

Let $X = W \GIT G \hookrightarrow V \GIT G$ be as above and $D$ a simple normal crossings divisor on $W \GIT G$ pulled back from $V \GIT G$. In this section we will define the moduli space of logarithmic quasimaps to $(X,D)$. The key point will be that although there is no genuine map from a curve to $X$, we have the \emph{induced morphism} \eqref{sncinduceeq} which contains all the data concerning tangency. It is this morphism which we will make logarithmic in order to compactify.

An alternative way to approach the problem is to equip the quotient stack containing $X$ with the divisorial logarithmic structure with respect to the closure of $D$ and form a moduli space of logarithmic maps to this quotient stack. This is essentially equivalent, but phrasing it as above allows us to take advantage of existing moduli spaces. Specifically, in this section we can take advantage of the moduli space of logarithmic maps to $[\AAA^r/\GG_m^r]$.

We will assume a familiarity with logarithmic geometry, for an introduction to the subject, see \cite{abramovich2013logarithmic,gross2011book}.

\subsection{Moduli space of logarithmic maps to $[\AAA^r/\GG_m^r]$}

\begin{definition}\label{mlog}
	Let $g,n$ be non-negative integers and let $\alpha$ be an $r \times n$ matrix of non-negative integers $\alpha = (\alpha_{i,j})_{i,j}$. A family of logarithmic maps to $[\AAA^r/\GG_m^r]$ over $(S,\cM_S)$, a fine and saturated logarithmic scheme, is a diagram
	
	\begin{equation*}
	\begin{tikzcd}
	C_S \arrow[d] \arrow[r] & {\argmr} \arrow[d] \\
	S \arrow[r]             & \mathrm{pt}           
	\end{tikzcd}
	\end{equation*}
	where $C_S \rightarrow S$ is an $n$-marked family of genus $g$, logarithmic curves and $C_S \rightarrow \argmr$ is a logarithmic morphism to $\argmr$ equipped with the divisorial logarithmic structure with respect to the coordinate hyperplanes, with contact orders $\alpha$ at the markings.
\end{definition}

The natural base of a family above is a (fine and saturated) logarithmic scheme rather than a scheme. In order to have a well behaved moduli stack over schemes, one instead restricts to logarithmic maps which are universal with respect to pullback, called \emph{minimal} logarithmic maps, see, for example, ~\cite[Proposition 4.1.1]{chen2014stable} and \cite{gillam2012logarithmic}. In the literature this is sometimes referred to as a \emph{basic} logarithmic map ~\cite[Definition 1.20]{gross2013logarithmic}.

\begin{theorem}[\hspace{-0.01em}{\cite[Proposition 3.1]{abramovich2018birational}}]
	Minimal logarithmic maps form a logarithmic algebraic stack $\frM_{g,\alpha}^{\log}([\AAA^r/\GG_m^r])$ . 
\end{theorem}

\begin{remark}
	The logarithmic structure on $\largmr$ is the divisorial logarithmic structure with respect to the complement of the locus of maps to $\argmr$ from smooth curves which hit the boundary in finitely many points. With respect to this logarithmic structure $\largmr$ is logarithmically smooth ~\cite[Proposition 1.6.1]{abramovich2018birational}.
\end{remark}

\subsection{Moduli space of logarithmic quasimaps}

Let $g,n$ be non-negative integers, let $\beta$ be an effective quasimap degree on $X= W \GIT G$ and let $D=D_1 + \dots + D_r$ be a s.n.c. divisor pulled back from $V \GIT G$, with the simplifying assumption that the intersection of any subsets of the components of $D$ is connected. Let $\alpha$ be an $r \times n$ matrix of non-negative integers $\alpha= (\alpha_{i,j})_{i,j}$ such that $\sum_{j} \alpha_{i,j} = D_{i} \cdot \beta$.
\begin{definition}\label{logquasimap}
 Define $\qlog$ as the fibre product of stacks 
	\begin{equation*}
	\begin{tikzcd}
	\qlog \arrow[d] \arrow[r] & {\largmr} \arrow[d, "\pi_2"] \\
	\cQ_{g,n}(X,\beta) \arrow[r, "\pi_1"]  & \frM_{g,n}(\argmr).                        
	\end{tikzcd}
	\end{equation*}
	Here $\cQ_{g,n}(X,\beta)$ is the moduli space of stable quasimaps and $\frM_{g,n}(\argmr)$ is the stack of maps from $n$-marked, genus $g$ prestable curves to $\argmr$. The morphism $\pi_1$ is given by associating to a stable quasimap the induced map to $\argmr$ \eqref{sncinduceeq}, and $\pi_2$ is given by forgetting the logarithmic structure. This is a fibre product in both the category of ordinary stacks and in the fine and saturated category. 
\end{definition}

\begin{remark}
	We equip $\qlog$ with the logarithmic structure defined by pullback via the morphism $\qlog \rightarrow \largmr$.
\end{remark}

\begin{proposition}
	The moduli stack $\qlog$ is a Deligne--Mumford stack.
\end{proposition}

\begin{proof}
	Since $\qlog$ is a fibre product of algebraic stacks it is automatically algebraic. The fact that it is a Deligne--Mumford stack will follow from the fact that the morphism $\pi_2$ is representable, which is true by ~\cite[Proposition 1.25]{gross2013logarithmic}. Given any cartesian diagram of algebraic stacks
	\begin{equation*}
	\begin{tikzcd}
	\cX \times_{\mathcal{Z}}\mathcal{Y} \arrow[d] \arrow[r] & \cY \arrow[d, "\pi_2"] \\
	\cX \arrow[r, "\pi_1"]                                  & \mathcal{Z}           
	\end{tikzcd}
	\end{equation*}
	with $\cX$ Deligne--Mumford and $\pi_2$ representable, then we can use the criterion ~\cite[Lemma 3.3.2]{cadman2007using} to show that any object $(x,y,\lambda)$, where $x \in \mathrm{Ob}(\cX), y \in \mathrm{Ob}(\cY), \lambda: \pi_1(x) \simeq \pi_2(y)$ must have finitely many automorphisms $(\varphi_x,\varphi_y) \in {\mathrm{Aut}_{\cX}(x) \times \mathrm{Aut}_{\cY}(y)}$. There are finitely many automorphisms, $\varphi_x$, as $\cX$ is Deligne--Mumford, but that fixes $\pi_2(\varphi_y) = \lambda \circ \pi_1(\varphi_x) \circ \lambda^{-1}$. On the other hand, $\pi_2$ is representable so it is an injection on automorphism groups, which determines $\varphi_y$.
\end{proof}

\begin{lemma}\label{finitetype}
	The moduli stack $\qlog$ is of finite type.
\end{lemma}
\begin{proof}
By ~\cite[proof of Proposition 3.2]{abramovich2018birational}, $\largmr$ is locally of finite type, so it suffices to show that $\qlog \rightarrow \largmr$ factors through a finite type substack of $\largmr$. But this follows from the fact that once we have fixed numerical data $g,n,\beta$, the number of components of the curve occurring in $\cQ_{g,n}(X,\beta)$ is bounded, as well as the total degree of each of the $r$ line bundles.
\end{proof}
\begin{proposition}
	The moduli space $\qlog$ is proper.
\end{proposition}
\begin{proof}
	Since $\cQ_{g,n}(X,\beta)$ is proper, to prove properness of $\qlog$ it suffices to prove properness of the morphism $\pi_2 : \largmr \rightarrow \frM_{g,n}(\argmr)$. Properness of $\pi_2$ follows from ~\cite[proof of Theorem 4.1]{gross2013logarithmic}. 
\end{proof}

\begin{proposition}\label{lemma : finite}
    The morphism $\qlog \rightarrow \cQ_{g,n}(X,\beta)$ is finite.
\end{proposition}
\begin{proof}
    It suffices to show that $\frM^{\log}_{g,\alpha}([\AAA^r/\GG_m^r]) \rightarrow \frM_{g,n}([\AAA^r/\GG_m^r])$ is finite. This follows from ~\cite[proof of Corollary 3.7]{abramovich2014stable}.
\end{proof}

The papers \cite{chen2014stable, abramovich2014stable} construct the moduli space of stable logarithmic maps by first constructing the moduli space in the case where $D$ is a smooth divisor, and then using this to build the moduli space when $D$ is s.n.c. The same approach would have worked in the quasimap setting.

\begin{lemma}\label{qlogsnc}
	Let $\alpha_i$ denote the $i^{\mathrm{th}}$ row of the $r \times n$ matrix $\alpha$. Then the moduli space $\qlog$ fits into a cartesian diagram (in the fine and saturated category)
	\begin{equation}\label{snclogquasimap}
	\begin{tikzcd}
	\qlog \arrow[d] \arrow[r] & {\cQ^{\log}_{g,\alpha_{1}}(X|D_1,\beta) \times \dots \times \cQ^{\log}_{g,\alpha_r}(X|D_r,\beta) } \arrow[d] \\
	\cQ_{g,n}(X,\beta) \arrow[r, "\Delta"]  & \cQ_{g,n}(X,\beta) \times \dots \times \cQ_{g,n}(X,\beta)                            
	\end{tikzcd}
	\end{equation}
	where the logarithmic structure on $\cQ_{g,n}(X,\beta)$ is pulled back from $\frM_{g,n}$.
\end{lemma}

\begin{proof}
	This is very similar to the argument in ~\cite[2.2]{battistella2022gromov}. Consider the diagram
	\begin{equation*}
	\begin{tikzcd}
	\qlog \arrow[r] \arrow[d]      & {\prod_{i=1}^r \frM_{g,\alpha_i}^{\log}(\agm) \times_{\frM_{g,n}^r} \frM_{g,n}} \arrow[d] \arrow[r] & {\prod_{i=1}^r\frM_{g,\alpha_i}^{\log}(\agm)} \arrow[d] \\
	{\cQ_{g,n}(X,\beta)} \arrow[r] & {\frM_{g,n}(\agm)^{r} \times_{\frM_{g,n}^r}\frM_{g,n}} \arrow[r] \arrow[d]                                                                             & {\frM_{g,n}(\agm)^{r}} \arrow[d]                                                              \\
	& {\frM_{g,n}} \arrow[r]                                                                                                                                 & {\frM_{g,n}^{r}}.                                                                             
	\end{tikzcd}
	\end{equation*}
	
Now, $\frM_{g,n}(\agm)^r \times_{\frM_{g,n}^r} \frM_{g,n} = \frM_{g,n}([\AAA^r/\GG_m^r])$ and by ~\cite[Theorem 2.6]{abramovich2014stable} we have that $\prod_{i=1}^r \frM_{g,\alpha_i}^{\log}(\agm) \times_{\frM_{g,n}^r} \frM_{g,n} = \frM_{g,\alpha}^{\log}([\AAA^r/\GG_m^r])$. Note that the latter is only true in the fine and saturated category. We can conclude that the top left square is (and so all squares are) cartesian, in the fine and saturated category, by the definition of $\qlog$ as in Definition \ref{logquasimap}. 

This tells us that the outer square in the following diagram is cartesian

\begin{equation}
 \begin{tikzcd}
\qlog \arrow[d] \arrow[r]                 & {\prod\limits_{i=1}^r \cQ_{g,\alpha_i}^{\log}(X|D_i,\beta)} \arrow[d] \arrow[r] & {\prod_{i=1}^r \frM_{g,\alpha_i}^{\log}(\agm)} \arrow[d] \\
{ \cQ_{g,n}(X,\beta)} \arrow[r, "\Delta"] & { \cQ_{g,n}(X,\beta)^r} \arrow[r]                                               & {\frM_{g,n}(\agm)^{r}}.                                 
\end{tikzcd}
\end{equation}

 Since the right hand square is also cartesian, the result follows.

\end{proof}


\begin{remark}\label{logepsilonquasimaps}
The entire construction could have equally been used to build a moduli space parametrising logarithmic $\epsilon$-quasimaps for any $\epsilon \in \QQ_{> 0}$. This would involve replacing $\cQ_{g,n}(X,\beta)$ with $\cQ^{\epsilon}_{g,n}(X,\beta)$ in the definitions. Moreover, the construction of the virtual fundamental class, Theorem \ref{logqvir}, also works for any $\epsilon$. We will not make use of this here but it will be important in any future application for wall-crossing.
\end{remark}

\section{Examples}\label{4}

Having built these moduli spaces, we now examine their geometry in a few simple examples and compare them with the corresponding moduli spaces in Gromov--Witten theory. We will examine a smooth divisor example, a simple normal crossings divisor example and a higher genus example, all for projective space targets. The first example is strictly speaking covered by \cite{battistella2021relative}, but we include it here anyway to show that even in the cases where the moduli space is as nice as possible, the quasimap moduli spaces are less complex.

\begin{example}\label{PNHlogquasimapspace}
Let $X= \PP^N$, let $D=H$ a hyperplane and suppose we are in genus $0$, degree $d$. We let $n=2$ be the number of markings and let $\alpha=(d,0)$ i.e. we are considering degree $d$ (quasi)maps to $\PP^N$ from rational curves with maximal tangency to the hyperplane $H$ at the first marking. We will compare $\cQ^{\log}_{0,(d,0)}(\PP^N|H,d)$ with $\overline{\cM}_{0,(d,0)}^{\log}(\PP^N|H,d)$. 
	
	$\overline{\cM}_{0,(d,0)}^{\log}(\PP^N|H,d)$ is irreducible of the expected dimension and we will see in Lemma \ref{PNHirred} that $\cQ_{0,(d,0)}^{\log}(\PP^N|H,d)$ is irreducible of the same dimension for the same reason. More concretely, both spaces (or rather their images in $\overline{\cM}_{0,2}(\PP^N,d)$ (resp. $\cQ_{0,2}(\PP^N,d)$) are the closures of the locus of \underline{\emph{maps}} $$(\PP^1,p_1,p_2) \rightarrow \PP^N$$ which are not mapped entirely into the hyperplane and hit $H$ with maximal tangency at $p_1$. Comparing these spaces with their images in the corresponding absolute spaces is a reasonable thing to do as the morphisms of moduli spaces here are finite (Lemma \ref{lemma : finite}) and generically injective. On the other hand, the quasimap spaces do not allow rational tails, curves with rational components with a single special point. Consequently, there will be fewer boundary divisors in $\cQ^{\log}_{0,(d,0)}(\PP^N|H,d)$ than in $\overline{\cM}_{0,(d,0)}^{\log}(\PP^N|H,d)$. In both cases the boundary divisors are \emph{comb loci}. These are loci where the source curve is of the form $C_0 \cup \dots \cup C_r$ with $C_i$ smooth, $C_i \cap C_0$ is a single nodal point  $i \neq 0$ and $C_i \cap C_j = \emptyset$ for $i,j \neq 0$. Moreover, $C_0$ contains the non-trivial tangency marking and gets mapped entirely into $H$, whereas $C_i$ only hits $H$ at the connecting node for $i \neq 0$.

	A comb locus in the quasimap space has domain curve containing at most two components, indeed this is implied by stability condition in Definition \ref{absquasimapdef}. There is a divisor corresponding to when $C=C_0$ which falls entirely into the hyperplane and $d-1$ distinct comb loci with two components $C=C_0 \cup C_1$ corresponding to the different possible degrees on each branch. However, in the Gromov--Witten space there can be up to $d+1$ components on the source curve of a comb locus. Each external component $C_i, \, \, i \neq 0$ must have positive degree and so there is a comb locus where each of the $d$ external components have degree 1 and the interior component carrying the maximal tangency marking is contracted. Below we include the number of boundary divisors in the quasimap and Gromov--Witten case to make the point that the quasimap spaces are much simpler.
	
	\begin{center}
		\begin{tabular}{||c c c||} 
			\hline
			$d$ & $\cQ_{0,(2,0)}^{\log}(\PP^N|H,d)$ & $\overline{\cM}_{0,(2,0)}^{\log}(\PP^N|H,d)$ \\ [0.5ex] 
			\hline\hline
			1 & 1 & 3  \\ 
			\hline
			2 & 2 & 7  \\
			\hline
			3 & 3 & 14  \\
			\hline
			4 & 4 & 26  \\
			\hline
			5 & 5 & 45  \\
			\hline
			6 & 6 & 75 \\
			\hline
		\end{tabular}
	\end{center}
\end{example}
These numbers are computed by enumerating genus-zero degree-weighted balanced tropical maps to $\RR_{\geq 0}$ which satisfy the appropriate stability condition. The connection to tropical maps can be found, for example in ~\cite[Section 2.5]{abramovich2020decomposition} and ~\cite[Theorem 1.1]{kennedy2022divisors}.
\begin{remark}
    In \cite{kennedy2022divisors}, a formula is determined for the number of boundary divisors of $\overline{\mathcal{M}}_{0,(d,0,\dots,0)}^{\log}(\PP^N|H,d)$, i.e. the logarithmic Gromov--Witten moduli space for $(\PP^N,H)$ with one maximal tangency marking and any number of redundant markings in any degree. 
\end{remark}
\begin{example}\label{p22linesQvsGW}
	We now move on to a simple normal crossings example. In the stable maps case this example comes from ~\cite[Section 1.2]{nabijou2022gromov} and the ideas come from ~\cite[Chapter 3]{nabijou2019recursion}. Let $X=\PP^2$ and $D=H_1 + H_2$ be the union of two hyperplanes. Let $g=0, d=2$ and let $\alpha$ be given by the matrix
	\begin{equation*}
	\begin{pmatrix}
	2 & 0 \\
	0 & 2
	\end{pmatrix}.
	\end{equation*}
	In other words we are considering degree $2$ (quasi)maps to $\PP^2$ with maximal tangency to $H_1$ at the first marking and maximal tangency to $H_2$ at the second marking. As in the previous example, $\cQ^{\log}_{0,\alpha}(\PP^2|H_1 +H_2,2)$ and ${\overline{\cM}^{\log}_{0,\alpha}(\PP^2|H_1 +H_2,2)}$ are irreducible of the same dimension but have differing numbers of boundary divisors. Instead of enumerating them we will make a different comparison. Recall that $\cQ^{\log}_{0,\alpha}(\PP^2|H_1 +H_2,2)$ is defined via the fibre product 
	\begin{equation*}
	\begin{tikzcd}
	{\cQ^{\log}_{0,\alpha}(\PP^2|H_1 +H_2,2)} \arrow[d] \arrow[r] & {\cQ^{\log}_{0,\alpha_2}(\PP^2|H_2,2)} \arrow[d] \\
	{\cQ^{\log}_{0,\alpha_1}(\PP^2|H_1,2)} \arrow[r]              & {\cQ_{0,2}(\PP^2,2)}                            
	\end{tikzcd}
	\end{equation*}
	where $\alpha_1 = (2,0)$ and $\alpha_2 = (0,2)$ (this is equivalent to \eqref{snclogquasimap}), in the category of fine and saturated logarithmic stacks. Similarly $\overline{\cM}^{\log}_{0,\alpha}(\PP^2|H_1 +H_2,2)$ \emph{can be} defined as the fibre product (in the fine and saturated category) \cite{abramovich2014stable}
	\begin{equation*}
	\begin{tikzcd}
	{\overline{\cM}^{\log}_{0,\alpha}(\PP^2|H_1 +H_2,2)} \arrow[d] \arrow[r] & {\overline{\cM}^{\log}_{0,\alpha_2}(\PP^2|H_2,2)} \arrow[d] \\
	{\overline{\cM}^{\log}_{0,\alpha_1}(\PP^2|H_1,2)} \arrow[r]              & {\overline{\cM}_{0,2}(\PP^2,2)}.                            
	\end{tikzcd}
	\end{equation*}
	Recall from the previous example that $\overline{\cM}^{\log}_{0,\alpha_i}(\PP^2|H_i,2) \rightarrow \overline{\cM}_{0,2}(\PP^2,2)$ is finite and generically injective. The (image of the) ordinary fibre product would be the intersection of (the images of) $\overline{\cM}^{\log}_{0,\alpha_1}(\PP^2|H_1,2)$ and $\overline{\cM}^{\log}_{0,\alpha_2}(\PP^2|H_2,2)$, each of which corresponded to the closure of the loci corresponding to maps from irreducible curves which don't map into $H_i$. So the image of the ordinary fibre product in $\overline{\cM}_{0,2}(\PP^2,2)$ is the intersection of the closure of these loci. On the other hand $\overline{\cM}^{\log}_{0,\alpha}(\PP^2|H_1 +H_2,2)$ is also irreducible and maps finitely and generically injectively into $\overline{\cM}_{0,2}(\PP^2,2)$. This image is also the closure of the locus of maps from irreducible curves which don't map entirely to either hyperplane. So comparing the ordinary fibre product and $\overline{\cM}^{\log}_{0,\alpha}(\PP^2|H_1 +H_2,2)$ amounts to the difference between the intersection of the closures and the closure of the intersection. In this case the ordinary fibre product is strictly larger than $\overline{\cM}^{\log}_{0,\alpha}(\PP^2|H_1 +H_2,2)$. By a dimension count $\overline{\cM}^{\log}_{0,\alpha}(\PP^2|H_1 +H_2,2)$ is $3$-dimensional. On the other hand, there is a locus in the ordinary fibre product corresponding to source curves of the form $C_0 \cup C_1 \cup C_2$, where $C_0$ is attached to $C_1$ and $C_2$ at nodes, both markings lie on $C_0$ and this component gets contracted to the point of intersection in $H_1 \cap H_2$. This locus has dimension 3 and forms the only other irreducible component of the ordinary fibre product. So in this case $\overline{\cM}^{\log}_{0,\alpha}(\PP^2|H_1 +H_2,2)$ is not the same as the ordinary fibre product.
	
	We could instead compare $\cQ^{\log}_{0,\alpha}(\PP^2|H_1 +H_2,2)$ with the ordinary fibre product of $\cQ^{\log}_{0,\alpha_1}(\PP^2|H_1,2)$ and $\cQ^{\log}_{0,\alpha_1}(\PP^2|H_2,2)$ over $\cQ_{0,2}(\PP^2,2)$. The dimension counts are all identical but due to the quasimap stability condition, even in the ordinary fibre product there can be no component with source curve $C_0 \cup C_1 \cup C_2$ with both markings on $C_0$, because $C_1$ and $C_2$ contain only a single special point. The consequence is that unlike in the stable maps case, $\cQ^{\log}_{0,\alpha}(\PP^2|H_1 +H_2,2)$ is the same underlying space as the ordinary fibre product.
	
	One reason for pointing out this difference is that the ordinary fibre product can be equipped with a `virtual' class leading to invariants ~\cite[Introduction, page 2]{nabijou2022gromov}. These invariants are equal to the orbifold invariants of the multi root stack given by rooting $\PP^2$ along $H_1$ and $H_2$, for sufficiently large rooting parameters \cite[Theorem B]{battistella2022local}. The orbifold Gromov--Witten invariants of a root stack coincide in genus-zero with the logarithmic Gromov--Witten invariants when the divisor is smooth \cite{abramovich2017relative}, but differ when the divisor is simple normal crossings. In the case above the difference can be attributed to the excess component in the ordinary fibre product. On the other hand there is no difference in the quasimap setting, which suggests that the difference between logarithmic and orbifold quasimap invariants may be less pronounced.
\end{example}
\begin{example}
    Let $X$ be a smooth projective toric variety associated to a fan $\Sigma$ and $D = \partial X$ be the full toric boundary. Then for any valid discrete data $g,n,\beta,\alpha$ we have that $\overline{\cM}^{\log}_{g,\alpha}(X|\partial X,\beta) = \cQ^{\log}_{g,\alpha}(X|\partial X,\beta)$. Stable maps without rational tails are stable quasimaps, and stable quasimaps without basepoints are stable maps, so it suffices to show that neither rational tails nor basepoints are possible here. Since any non-constant morphism from a rational curve must hit the boundary in at least two points we must have that this rational component contains at least two special points, so a stable logarithmic map contains no rational tails. Now we show that any stable logarithmic quasimap cannot have a basepoint. Any basepoint of a quasimap occurs when, for any \emph{primitive collection} ~\cite[Definition 5.1.5]{cox2011toric} of $\Sigma$, the corresponding  sections of the line bundles for the toric boundary divisors vanish. For example, if $X=\PP^N$ then the only primitive collection is the union of all the rays, and a basepoint occurs only when \emph{all} sections defining the quasimap vanish simultaneously. Any such primitive collection consists of rays where the intersection of the corresponding toric divisors is empty. In the case where $X=\PP^N$ the intersection of all coordinate hyperplanes is clearly empty. Consequently, the curve component containing any basepoint would have at least one section corresponding to a ray in the primitive collection which doesn't vanish identically. On the other hand, a basepoint must occur at a point where this section vanishes, but since these points are either marked or nodes they cannot be basepoints.
\end{example}

\section{Logarithmic quasimap invariants}\label{5}

Let $\qlog$  be the moduli space from Definition \ref{logquasimap}. We will produce a virtual fundamental class on this moduli space using the construction of Behrend and Fantechi \cite{behrend1997intrinsic}.
\begin{lemma}\label{logmap2stack}
	Recall $X = W \GIT G \hookrightarrow [W/G]$. The divisor $D$ induces a logarithmic structure on the quotient stack $[W/G]$ pulled back from the toric logarithmic structure on $[\AAA^r/\GG_m^r]$. Moreover, a logarithmic quasimap to $(X,D)$ from a curve $C$ induces a logarithmic structure on $C$ and a logarithmic morphism $C\rightarrow [W/G]$.
\end{lemma}
\begin{proof}
	We have a morphism $$[W/G]\rightarrow [\AAA^r/\GG_m^r]. $$
	Moreover, this morphism is strict, so the map from the underlying curve of $C$ to $[W/G]$ together with a logarithmic morphism $$C \rightarrow [\AAA^r/\GG_m^r]$$ defines a logarithmic morphism to $[W/G]$.
\end{proof}

 Therefore we have a diagram
\begin{equation*}
\begin{tikzcd}
\cC_{\cQ} \arrow[d,"\pi"] \arrow[r] \arrow[rr, "u", bend left] & \cC_{\frM} \arrow[d] \arrow[rr, bend right] & {[W/G]} \arrow[r] & {[\AAA^r/\GG_m^r]} \\
\qlog \arrow[r]                                          & {\frM_{g,\alpha}^{\log}([\AAA^r/\GG_m^r])}   &                   &                   
\end{tikzcd}
\end{equation*}
Here $\cC_{\cQ}$ and $\cC_{\frM}$ are the universal curves over the moduli spaces $\qlog$ and $\frM^{\log}_{g,\alpha}([\AAA^r/\GG_m^r])$, $\pi$ is the projection $\cC_{\cQ} \rightarrow \qlog$ and $u$ is the universal map defined in Lemma \ref{logmap2stack}. Furthermore, the square is cartesian in both the fine, saturated category and the ordinary category. We impose in this section that the morphism $[W/G] \rightarrow [\AAA^r/\GG_m^r]$ be flat, used in the proof of Proposition \ref{obstructiontheory}, but we expect this assumption to usually hold.

\subsection{Obstruction theory} In \cite{olsson2003logarithmic} the author introduces the stack of logarithmic structures. Namely, if $(S,\cM_S)$ is a fine logarithmic scheme, then there is a fibered category over $(Sch)/S$, called $\mathrm{Log}_{(S,\cM_S)}$, where the objects over a morphism of schemes $T \rightarrow S$ is an enhancement of this morphism to fine logarithmic schemes $(T,\cM_T) \rightarrow (S,\cM_S)$. The main result of \cite{olsson2003logarithmic} is that $\mathrm{Log}_{(S,\cM_S)}$ is an algebraic stack, locally of finite presentation. Using this, Olsson defines the logarithmic cotangent complex as follows.

\begin{definition}\cite[Definition 3.2]{olsson2005logarithmic}
	Let $X \rightarrow Y$ be a morphism of fine logarithmic schemes. Let $X \rightarrow \mathrm{Log}_{Y}$ be the induced morphism. Define $\mathbb{L}^{\log}_{X/Y}$, the \emph{logarithmic cotangent complex}, to be the ordinary cotangent complex of this morphism.
\end{definition}

We will need an extension of this theory to certain algebraic stacks. In \cite{gross2013logarithmic}, the authors did this for Deligne--Mumford stacks by reducing to an étale cover. Our situation is slightly more general and we will just define the logarithmic cotangent complex in the same way.

Let $\LL^{\log}_{[W/G]}$ denote the \emph{logarithmic cotangent complex} of $[W/G]$. This is defined as the ordinary cotangent complex of the morphism $$[W/G] \rightarrow \mathrm{Log}_{\CC}$$  where the latter denotes the stack of logarithmic structures for $\Spec \CC$ with the trivial logarithmic structure. 

\begin{remark}
	The morphism $[W/G] \rightarrow \mathrm{Log}_{\CC}$ factors through the morphism $[W/G] \rightarrow [\AAA^r/\GG_m^r]$. Moreover, since $[\AAA^r/\GG_m^r]$ is logarithmically étale, the logarithmic cotangent complex is equal to the cotangent complex of $$[W/G] \rightarrow [\AAA^r/\GG_m^r].$$
\end{remark}

\begin{remark}
    In our setup, in the case where the intersection of all of the divisor components is non-empty, then $[\AAA^r/\GG_m^r]$ is nothing more than the artin fan of $X$ \cite{abramovich2018birational}. In general, in our setup, there is a strict morphism from the artin fan to $[\AAA^r/\GG_m^r]$.
\end{remark}
\begin{lemma}
	There is a morphism in the derived category of $\qlog$ $$\phi : \mathbf{R} \pi_{*}(\mathbf{L}u^* \LL^{\log}_{[W/G]} \otimes \omega_{\pi}) \rightarrow \LL_{\qlog/\frM^{\log}_{g,\alpha}([\AAA^r/\GG_m^r])}.$$
\end{lemma}
\begin{proof}
	By the functoriality properties of the cotangent complex ~\cite[Theorem-Definition 17.3 (2)]{gerard00champs}, we get a morphism in the derived category $$\mathbf{L}u^* \LL^{\log}_{[W/G]} \rightarrow \LL_{\cC_{\cQ}/\cC_{\frM}}.$$
	On the other hand, base change properties of the cotangent complex ~\cite[Theorem-Definition 17.3 (4)]{gerard00champs} imply that $$\LL_{\cC_{\cQ}/\cC_{\frM}} \cong \mathbf{L}\pi^* \LL_{\cQ/\frM}$$
	where $\cQ = \qlog$ and $\frM= \frM^{\log}_{g,\alpha}(\agm)$.
	Tensoring this morphism with the relative dualising sheaf gives $$\mathbf{L}u^* \LL^{\log}_{[W/G]} \otimes \omega_{\pi} \rightarrow \mathbf{L}\pi^*\LL_{\cQ/\frM} \otimes \omega_{\pi}$$
	so, by applying $\mathbf{R}\pi_{*}$ and using adjunction, we get a morphism in the derived category 
	$$\phi: \mathbf{R} \pi_{*}(\mathbf{L}u^* \LL^{\log}_{[W/G]} \otimes \omega_{\pi}) \rightarrow \LL_{\cQ/\frM}.\qedhere$$
\end{proof}

Now we need to show two things:

\begin{enumerate}
	\item $\phi$ defines an obstruction theory.
	\item $\cE^{\bullet} := \mathbf{R} \pi_{*}(\mathbf{L}u^* \LL^{\log}_{[W/G]} \otimes \omega_{\pi})$ is of perfect amplitude contained in $[-1,0]$.
\end{enumerate}

\begin{proposition}\label{obstructiontheory}
	The morphism $\phi$ defines an obstruction theory.
\end{proposition}

\begin{proof}
	This argument follows exactly as in ~\cite[Proposition 5.1]{gross2013logarithmic}. Let $T \hookrightarrow \bar{T}$ be a square zero extension with ideal $J$ and let $h : T \rightarrow \cQ$ be a morphism. Endow $T$ and $\bar{T}$ with logarithmic structures pulled back from $\frM$ and pullback the universal curve to $T$ and $\bar{T}$, which we will denote $\cC_T$ and $\cC_{\bar{T}}$. We have a commutative diagram
	
	\begin{equation}
\begin{tikzcd}
                             & \cC_T \arrow[rr, "\tilde{h}"] \arrow[ld, "p"] \arrow[dd] \arrow[rrrr, "u_T", bend left] &                & \cC_{\cQ} \arrow[dd] \arrow[ld, "\pi"] \arrow[rr, "u"] &  & {[W/G]} \arrow[dd] \\
T \arrow[rr, "\quad h"] \arrow[dd] &                                                                                         & \cQ \arrow[dd] &                                                        &  &                    \\
                             & \cC_{\overline{T}} \arrow[ld] \arrow[rr]                                                &                & \cC_{\frM} \arrow[ld] \arrow[rr]                             &  & {[\AAA^r/\GG_m^r]} \\
\overline{T} \arrow[rr]      &                                                                                         & \frM           &                                                        &  &                   
\end{tikzcd}.
	\end{equation}
	
	Recall from ~\cite[III 2.2.4]{illusie71complexe} and ~\cite[2.21]{olsson2006deformation} that $h$ extends to a morphism $\bar{T} \rightarrow \cQ$ if and only if $\omega(h) \in \mathrm{Ext}^1(\mathbf{L}h^* \LL_{\cQ/\frM},J)$ is $0$ where $\omega(h)$ is defined by $$\mathbf{L}h^*\LL_{\cQ/\frM} \rightarrow \LL_{T/\bar{T}} \rightarrow \tau_{\geq -1}\LL_{T/\bar{T}} = J[1]$$
	To show that $\phi$ defines an obstruction theory we use ~\cite[Proposition 4.53]{behrend1997intrinsic}. We show that an extension exists if and only if $\phi^*\omega(h) = 0$ and moreover if an extension exists, the set of isomorphism classes of extensions form a torsor under $\Hom(\mathbf{L}h^* \cE^{\bullet},J)$.
	As in ~\cite[Proposition 5.1]{gross2013logarithmic}, a lift exists if and only if $u_T$ extends logarithmically to $\cC_{\bar{T}}$. But using ~\cite[Theorem 5.9]{olsson2005logarithmic} there is a class $o \in \mathrm{Ext}^1(\mathbf{L}u_{T}^*\LL^{\log}_{[W/G]},p^*J)$ which vanishes if and only if there is a lift. Moreover, the lifts form a torsor under $\Hom(\mathbf{L}u_{T}^*\LL^{\log}_{[W/G]},p^*J)$. However, just as in \cite{gross2013logarithmic} we have 
	\begin{align*}
	&\mathrm{Ext}^k(\mathbf{L}h^*\mathbf{R}\pi_{*}(\mathbf{L}u^*\LL^{\log}_{[W/G]}\otimes \omega_{\pi}),J)& \\
	= \, &\mathrm{Ext}^k(\mathbf{L}u^*\LL^{\log}_{[W/G]}\otimes \omega_{\pi},  \mathbf{L}\pi^!\mathbf{R}h_{*}J) \\
	= \, &\mathrm{Ext}^k(\mathbf{L}u^*\LL^{\log}_{[W/G]}, \mathbf{R}\tilde{h}_{*}p^*J)	\\
	= \, &\mathrm{Ext}^k(\mathbf{L}u_{T}^*\LL^{\log}_{[W/G]}, p^*J)
	\end{align*}
	which, when $k=1$, sends $\phi^*\omega(h)$ to the obstruction class $o$.
\end{proof}

\begin{remark}\label{liftinggeneralisation}
	We have used results from \cite{olsson2005logarithmic}, in particular ~\cite[Axiom 1.1 (ii), (iv), Theorem 5.9]{olsson2005logarithmic}, which are proved in the context of logarithmic schemes. We require these results to extend to certain algebraic stacks. The same proofs go through without change, because the results follow from the corresponding properties of the ordinary cotangent complex, which hold for such algebraic stacks.
 An alternative strategy would have been to show that the morphism $$\cQ_{g,n}(X,\beta) \rightarrow \frM_{g,n}([\AAA^r/\GG_m^r])$$ admits an obstruction theory as in \cite[Proposition 3.1.1]{abramovich2014comparison}, which would have allowed us to not invoke the deformation theory of logarithmic maps.
\end{remark}
\subsection{Perfectness of the obstruction theory}

\begin{proposition}\label{perfect}
	$\mathbf{R} \pi_{*}(\mathbf{L}u^* \LL^{\log}_{[W/G]} \otimes \omega_{\pi})$ is of perfect amplitude contained in $[-1,0]$.
\end{proposition}

	\begin{definition}
	Define the \emph{logarithmic tangent complex} $\TT^{\log}_{[W/G]}$ as the derived dual of $\LL^{\log}_{[W/G]}$.
\end{definition}

\begin{lemma}\label{Lemma : cohomology of Tlog}
	The logarithmic tangent complex $\TT^{\log}_{[W/G]}$ has cohomology supported in $[0,1]$.
\end{lemma}

\begin{proof}
We first assume prove the lemma in the case where $W=V$ is a vector space. In this situation, the divisor $D$ comes from a morphism $[V/G] \rightarrow [\AAA^r/\GG_m^r]$ defined by sections $s_1,\dots, s_r$ of the trivial line bundle
\begin{align*}
S:[V/G] &\rightarrow [\AAA^r/\GG_m^r] \\
\underline{v} &\mapsto \left(s_1(\underline{v}),\dots,s_{r}(\underline{v})\right).
\end{align*}
There is a distinguished triangle in the derived category of $[V/G]$ $$\mathbf{L}S^*\LL_{[\AAA^r/\GG_m^r]} \rightarrow \LL_{[V/G]} \rightarrow \LL^{\log}_{[V/G]} \rightarrow \mathbf{L}S^*\LL_{[\AAA^r/\GG_m^r]}[1].$$

By \cite{behrend2005derham} the \emph{tangent complex} of $[V/G]$ is given in degrees $[-1,0]$ by the differential of the action.
\begin{align}\label{tangentcomplexVG}
\mathfrak{g} \otimes \cO_V &\rightarrow \cO_{V}^{\oplus \dim V}
\end{align}

	
	

The same reasoning tells us that the tangent complex of $[\AAA^r/\GG_m^r]$ is given in degrees $[-1,0]$ by 
\begin{align*}
\cO_{\AAA^r}^{\oplus r} &\rightarrow \cO_{\AAA^r}^{\oplus r} \\
e_i &\mapsto (0, \dots, 0,x_i,0, \dots,0).
\end{align*}

Pulling this back to $[V/G]$ gives the complex
\begin{align*}
\cO_{V}^{\oplus r} &\rightarrow \cO_{V}^{\oplus r} \\
e_i &\mapsto (0,\dots, s_i(\underline{v}) ,\dots,0).
\end{align*}

Taking the dual distinguished triangle	
	\begin{equation}
	\TT_{[V/G]}^{\log} \rightarrow \TT_{[V/G]} \rightarrow \mathbf{L}S^*\TT_{[\AAA^r/\GG_m^r]} \rightarrow \TT^{\log}_{[V/G]}[1]
	\end{equation}
	tells us that we can build $\TT^{\log}_{[V/G]}$ as the shift of the mapping cone of $\TT_{[V/G]} \rightarrow \mathbf{L}S^*\TT_{[\AAA^r/\GG_m^r]}$. Since the morphism \eqref{tangentcomplexVG} is injective it follows that the mapping cone actually has cohomology supported in $[-1,0]$ and so after shifting, $\TT^{\log}_{[V/G]}$ has cohomology supported in $[0,1]$. This concludes the proof of the lemma when $W=V$. Now we want to show that if we have $W \hookrightarrow V$ then $\TT_{[W/G]}^{\log}$ also has cohomology supported in $[0,1]$. Since we are only considering divisors pulled back from $[V/G]$ we have morphisms 
\begin{equation}
	\begin{tikzcd}
	{[W/G]} \arrow[r, "i"] \arrow[rr, "S \circ i", bend left] & {[V/G]} \arrow[r, "S"] & {[\AAA^r/\GG_m^r]}.
	\end{tikzcd}
\end{equation}
	The associated (dual) distinguished triangle is
	\begin{equation}
		\TT_{[W/G]/[V/G]} \rightarrow \TT^{\log}_{[W/G]} \rightarrow \mathbf{L}i^*\TT_{[V/G]}^{\log} \rightarrow \TT_{[W/G]/[V/G]}[1]
	\end{equation}
	
	But by the cartesian diagram 
\begin{equation*}
	\begin{tikzcd}
	W \arrow[d] \arrow[r]   & V \arrow[d] \\
	{[W/G]} \arrow[r, "i"'] & {[V/G]}    
	\end{tikzcd}
\end{equation*}
together with the fact that $W$ has only l.c.i. singularities (containined in $W\setminus W^s$) and $V \rightarrow [V/G]$ is flat, we know that $\TT_{[W/G]/[V/G]}$ has cohomology supported in degree $1$ ~\cite[Section 4.5]{ciocan2014stable}. The distinguished triangle tells us we can build $\TT^{\log}_{[W/G]}$ as the mapping cone of the morphism $\mathbf{L}i^*\TT_{[V/G]}^{\log}[-1] \rightarrow \TT_{[W/G]/[V/G]}$ which by the result above must have cohomology supported in $[0,1]$.
\end{proof}

We can also conclude that $\mathbf{L}u^* \TT^{\log}_{[W/G]}$ also has cohomology supported in $[0,1]$. Furthermore, we know that $\mathbf{R} \pi_{*}(\mathbf{L}u^* \LL^{\log}_{[W/G]} \otimes \omega_\pi)$ is quasi-isomorphic to $\left(\mathbf{R} \pi_{*}(\mathbf{L}u^* \TT^{\log}_{[W/G]})\right)^{\vee}$, see for example ~\cite[4.1]{fausk2003isomorphisms}.

\begin{proof}[Proof of Proposition \ref{perfect}]
	We need to show that $\mathbf{R} \pi_{*}(\mathbf{L}u^* \TT^{\log}_{[W/G]})$ is of perfect amplitude contained in $[0,1]$. First we show that $\mathbf{R} \pi_{*}(\mathbf{L}u^* \TT^{\log}_{[W/G]})$ has cohomology supported in $[0,1]$. This will follow from the argument of ~\cite[Theorem 4.5.2]{ciocan2014stable}. Let $C$ be a geometric fibre of $\cC_{\cQ}$, corresponding to a geometric point $i_p : \Spec \CC \hookrightarrow \cQ^{\log}_{g,\alpha}(X|D,\beta)$. Let $F^{\bullet}$ denote the restriction of $\mathbf{L}u^* \TT^{\log}_{[W/G]}$ to $C$ so that $$\mathbf{L}i_{p}^{*}\mathbf{R} \pi_{*}\left(\mathbf{L}u^* \TT^{\log}_{[W/G]}\right) \cong \mathbf{R}\Gamma (F^{\bullet})$$ in the derived category. We show that $\mathbf{R}^2\Gamma(F^{\bullet})=0$. We have that $H^1(F^{\bullet})$ is a torsion sheaf. This follows from the fact that away from the set $B$ of finitely many basepoints on $C$ the morphism $u$ factors through $X$, and the restriction $F^{\bullet}|_{C\setminus B}$ is quasi-isomorphic to the vector bundle $\cT^{\log}_X$. We have the spectral sequence $$E^{p,q}_{2} = \mathbf{R}^p \Gamma(H^q(F^{\bullet})) \Rightarrow \mathbf{R}^{p+q}\Gamma(F^{\bullet}) = E^{p+q}.$$
	The second page looks like
	\begin{equation*}
	\begin{tikzcd}
	{H^1(C,H^0(F^{\bullet}))} & {H^1(C,H^1(F^{\bullet}))} \\
	{H^0(C,H^0(F^{\bullet}))} & {H^0(C,H^1(F^{\bullet}))}
	\end{tikzcd}
	\end{equation*}
	but by the above we know that $H^1(C,H^1(F^{\bullet}))=0$. Consequently $R^2\Gamma(F^{\bullet}) = 0$. 
	
	Using the criterion of  ~\cite[3.6.4]{illusie2005grothendieck}, the complex $\mathbf{R} \pi_{*}(\mathbf{L}u^* \TT^{\log}_{[W/G]})$, which is cohomologically supported in $[0, 1]$, is of perfect amplitude contained in $[0, 1]$ if and only if for every point $p$ we have $$H^{-1}\left(\mathbf{L}i_{p}^{*}\mathbf{R} \pi_{*}\left(\mathbf{L}u^* \TT^{\log}_{[W/G]}\right)\right) = H^{-1}(\mathbf{R}\Gamma(F^{\bullet})) = 0$$
	But this is also clear by the above.
\end{proof}

\begin{theorem}\label{logqvir}
	There is a relative perfect obstruction theory on $\qlog$ over $\frM^{\log}_{g,\alpha}([\AAA^r/\GG_m^r])$ leading to a virtual fundamental class $[\qlog]^{\vir}$.
\end{theorem}

\begin{proof}
 Combining Proposition \ref{obstructiontheory} with Proposition \ref{perfect}, and since $\frM^{\log}_{g,\alpha}([\AAA^r/\GG_m^r])$ is irreducible, we get a virtual fundamental class on $\qlog$ by \cite{manolache12virtual}.  
\end{proof}

\begin{remark}
    In the proof of Proposition \ref{perfect}, we have crucially used the assumption that the divisor is pulled back from $V \GIT G$. There is an analogy with the proof of perfectness of the obstruction theory in ~\cite[Section 4.5]{ciocan2014stable} in the case when $W$ has l.c.i. singularities, where the embedding $W \hookrightarrow V$ is also used.
\end{remark}

\subsection{Toric Divisors}\label{Section : Toric Divisiors}

On the other hand Proposition \ref{perfect} becomes significantly simpler in the case where $X$ is a toric variety and $D = D_1 + \dots + D_r$ is a subset of the toric boundary.

Let $X$ be a smooth projective toric variety associated to some fan $\Sigma$. Recall we have the induced quotient description $X = \AAA^M \GIT \GG_m^s$, see Remark \ref{toric presentation}. The weights of the action can be encoded in a matrix 
\begin{equation*}
	\begin{pmatrix}
	a_{11} & a_{12} & \dots & a_{1M} \\
	a_{21} & a_{22} & \dots & a_{2M} \\
	\vdots & \ddots & \ddots & \vdots \\
	a_{s1} & a_{s2} & \dots    & a_{sM}
	\end{pmatrix}
\end{equation*}

which can be read off from the second morphism in the exact sequence 
\begin{equation}
0 \rightarrow M \rightarrow \ZZ^{|\Sigma(1)|} \rightarrow \Pic(X) \rightarrow 0.
\end{equation}
Here $M$ is the character lattice of the torus. The first morphism is given by the matrix whose rows are the rays of the fan $\{\rho_{i}\}_{i=1}^M$. If $D$ is any smooth divisor then by ~\cite[Proposition 2.1]{cox1995homogeneous} there is an isomorphism 
\begin{equation}\label{coxiso}
H^0(X,\cO_{X}(D)) = \left\langle \prod_{i=1}^M x_{i}^{b_i} : \sum_i b_i [D_{\rho_i}] = [D] \right\rangle
\end{equation}
where $D_{\rho_i}$ is the toric boundary divisor associated to the ray $\rho_i$. Therefore the section $s_D \in H^0(X,\cO_X(D)$ cutting out $D$ has a well defined \emph{degree} 

   \[   \begin{pmatrix}
         \deg_{1}s_D \\
         \deg_{2}s_D \\ 
         \vdots\\ 
        \deg_{s}s_D
     \end{pmatrix}
     =
     \begin{pmatrix}
         a_{11} & a_{12} & \cdots & a_{1M}\\
         a_{21} & a_{22} & \cdots & a_{2M}\\ 
         \vdots & \vdots & \ddots & \vdots\\ 
         a_{s1} & a_{s2} & \cdots & a_{sM} 
     \end{pmatrix}
     \times
     \begin{pmatrix}
         b_{1} \\
         b_{2}\\ 
         \vdots \\ 
         b_{M} 
     \end{pmatrix} \]

for any $(b_1,\dots,b_M)$ appearing as the monomial powers in \eqref{coxiso}. This degree is equivalently given by the grading in the homogeneous coordinate ring of $X$ \cite{cox1995homogeneous}.

\begin{lemma}\label{toricbdyvfc}
	Let $X$ be a smooth projective toric variety associated to some fan $\Sigma$ inducing a quotient description $X = \AAA^M \GIT \GG_m^s$. If $D = D_1 + \dots + D_r$ is a subset of the toric boundary corresponding to rays $\rho_{j_1},\dots,\rho_{j_r} \in \Sigma(1)$ (set $J= \{j_1,\dots,j_r\}$), then there is a \emph{logarithmic Euler sequence} 
	\begin{equation}\label{logeuler}
	0 \rightarrow \cO_{X}^{\oplus s} \rightarrow \bigoplus_{\Sigma(1)\setminus \rho_J} \cO_{X}(D_{\rho}) \bigoplus_{J} \cO_{X} \rightarrow \cT^{\log}_{X} \rightarrow 0
	\end{equation}
	Moreover, in this situation $\TT^{\log}_{[\AAA^M/\GG_m^s]}$ is quasi-isomorphic to a sheaf.
\end{lemma}
\begin{proof}
	When $X$ is toric the morphism $\TT_{[\AAA^M/\GG_m^s]} \rightarrow \mathbf{L}S^*\TT_{[\AAA^r/\GG_m^r]}$ becomes very explicit, given by the following diagram
	
		\begin{equation}\label{morphismcotangent}
		\begin{tikzcd}
		\cO_{\AAA^M}^{\oplus s} \arrow[rr, bend left, "{e_i \mapsto(a_{i,1}x_1,\dots,a_{i,M}x_M)}"] \arrow[dd, "e_i \mapsto (\deg_i s_j)_j "'] &  & \cO_{\AAA^M}^{\oplus M} \arrow[dd, "e_i \mapsto \left(\frac{\partial s_j}{\partial x_i}\right)_j "] \\
		&                                                                                                                                                                                                                                 \\
		\cO_{\AAA^M}^{\oplus r} \arrow[rr, bend right, "e_i \mapsto s_i e_i"]                                                                               &  & \cO_{\AAA^M}^{\oplus r}                                                                                                                                                                                                          
		\end{tikzcd}
		\end{equation}
	
	If $D = D_1 + \dots + D_r$ is a subset of the toric boundary then $s_1,\dots s_r$ are just $x_{j_1},\dots,x_{j_r}$ corresponding to the homogeneous coordinates on $X$. But since the right-hand vertical map is given by differentiating the sections this map becomes $e_i \mapsto (0,\dots,1,\dots,0)$ if $i=j_{k}$ for some $k$ with $1$ in the $k^{\mathrm{th}}$ place and $e_i \mapsto (0,\dots,0)$ otherwise. Consequently, by forming the mapping cone (and shifting) we have that $\TT^{\log}_{[\AAA^M/\GG_m^s]}$ is given by the three term complex 
	\begin{equation*}
	\cO_{\AAA^M}^{\oplus s} \rightarrow \cO_{\AAA^M}^{\oplus M} \oplus \cO^{\oplus r}_{\AAA^M} \rightarrow \cO_{\AAA^M}^{\oplus r}
	\end{equation*}
	But the right-hand map is now surjective so this complex is quasi-isomorphic to a two term complex 
	\begin{equation}
	\cO^{\oplus s}_{\AAA^M} \rightarrow \cO_{\AAA^M}^{\oplus M - r} \oplus \cO^{\oplus r}_{\AAA^M}
	\end{equation}
	Where in the second term the first ${M-r}$ copies have action according to the weight matrix and the latter $r$ copies have the trivial action. Since the first map is injective the second part follows. For the first part we just pullback this complex under the inclusion $X \hookrightarrow [\AAA^M/\GG_m^s]$. 
\end{proof}

\begin{corollary}\label{toricperfect}
	If $D$ is a subset of the toric boundary then the complex $(\mathbf{R} \pi_{*} \mathbf{L}u^* \TT^{\log}_{[\AAA^M/\GG_m^s]})^{\vee}$ is of perfect amplitude contained in $[-1,0]$.
\end{corollary}
\begin{corollary}
	There is a relative perfect obstruction theory on $\qlog$ over $\frM^{\log}_{g,\alpha}([\AAA^r/\GG_m^r])$ leading to a virtual fundamental class $[\qlog]^{\vir}$.
\end{corollary}

\section{Comparison with Battistella--Nabijou theory}\label{BN}
In \cite{battistella2021relative} the theory of relative quasimaps is developed in the case where $X$ is a smooth projective toric variety, $D \subset X$ is a smooth, very ample divisor (not necessarily toric) and in genus-zero. This is achieved by mimicking the relative Gromov--Witten construction in \cite{gathmann2002absolute} and so the relative quasimap moduli space is defined as a closed substack of the space of absolute quasimaps.

Let $X$ be a smooth projective toric variety associated to a complete fan $\Sigma$ and $D$ a smooth (very ample) divisor, cut out by a section $s_D \in H^0(X,\cO_{X}(D))$. Recall that  given an ordinary quasimap from a (marked) curve $C$ to $X$ there is an \textit{induced} line bundle section pair $(L_D,u_D)$ and in the case where there are no basepoints these are just the pullbacks of $\cO_{X}(D)$  and $s_D$ respectively along the map $C \rightarrow X$.

\begin{definition}\label{def : relquasimaps}[\cite{battistella2021relative}] Let $n \geq 2$ be the number of marked points, let $\beta$ be an effective curve class and $\alpha = (\alpha_1, \dots, \alpha_n)$ such that $\sum_{i} \alpha_i \leq D \cdot \beta$. Define the \textit{moduli space of relative quasimaps} $\cQ^{\mathrm{rel}}_{0,\alpha}(X|D,\beta)$ to be the locus of absolute quasimaps $\left((C, p_1, \dots, p_n), \{L_{\rho}, u_{\rho}\}_{\rho \in \Sigma(1)}, \{c_m\}_{m \in M} \right)$ in $\cQ_{0,n}(X,\beta)$ such that for every $Z$, a connected component of $u_{D}^{-1}(0)$ we have
	\begin{enumerate}\label{Gathdition}
		\item If $Z$ is a point, then either it is unmarked or a marked point $p_i$ such that the multiplicity of $f$ at $p_i$ is at least $\alpha_i$.
		\item If $Z$ is one dimensional let $C^{(i)}$ for $1 \leq i \leq r$ be the irreducible components of $C$ not in $Z$, but intersecting $Z$, and let $m^{(i)}$ be the multiplicity of $u_D|_{C^{(i)}}$ at the node $C^{(i)} \cap Z$ along $D$. Then we must have  $\deg L_{D}|_Z + \sum\limits_{i} m^{(i)} \geq \sum\limits_{p_i \in Z} \alpha_i$.
	\end{enumerate}
\end{definition}
\begin{remark}\label{toricquotientstack}
	The definition of absolute quasimap here is taken from \cite{ciocan2010moduli}. If we write the toric variety $X$ as a GIT quotient $\AAA^{M} \GIT\GG_m^s$ as prescribed by the fan $\Sigma$, then this definition is equivalent to the Definition \ref{absquasimapdef}, involving a morphism $C \rightarrow [\AAA^M/\GG_m^s]$. 
\end{remark}
\begin{remark}
	let $X = \PP^N$ and let $D = H \cong \PP^{N-1}$ be the hyperplane given in coordinates by $\{x_0 = 0\}$. Then $\cQ_{0,\alpha}^{\mathrm{rel}}(\PP^N|H,d)$ is irreducible of dimension $\dim\cQ_{0,n}(\PP^N,d) - \sum\limits_i \alpha_i$ and so has an actual fundamental class with which one can define relative quasimap invariants. For the general case of $(X,D)$, note that $\cO_{X}(D)$ defines a map $j : X \hookrightarrow \PP^N$. Battistella and Nabijou show that the following diagram is cartesian (with $d= j_{\ast} \beta$)
	
	\begin{equation*}	
	\begin{tikzcd}
	{\mathcal{Q}^{\mathrm{rel}}_{0,\alpha}(X|D,\beta)} \arrow[d, hook] \arrow[r, "i'"] & {\mathcal{Q}^{\mathrm{rel}}_{0,\alpha}(\mathbb{P}^N|H,d)} \arrow[d, hook] \\
	{\cQ_{0,n}(X,\beta)} \arrow[r, "j'"]           & {\cQ_{0,n}(\mathbb{P}^N,d)}.                       
	\end{tikzcd}
	\end{equation*}
	Then they use diagonal pull-back to define a virtual fundamental class on $\cQ_{0,\alpha}^{\mathrm{rel}}(X|D,\beta)$ and define relative quasimap invariants.
\end{remark}
\begin{remark}\label{equality}
	From now on we restrict to the case where $\sum_i \alpha_i = D \cdot \beta$ in which case the inequalities in Definition \ref{def : relquasimaps} (2), become equalities.
\end{remark}
\begin{theorem}\label{eq}
	Let $g=0$ and $D \subset X$ be a smooth, very ample divisor inside a smooth projective toric variety. There is a morphism $g: \qlogo \rightarrow \qrel $ and $$g_{\ast}[\qlogo]^{\mathrm{vir}} = [\qrel]^{\vir}.$$
\end{theorem}

The strategy for proving Theorem \ref{eq} is as follows. We will first prove the existence of the morphism $g$. Then we prove Theorem \ref{eq} for the case of $X=\PP^N$ and $D=H$ a hyperplane. We will then use the ampleness condition to pullback this result to the general case.

\begin{proposition}\label{log2rel}
	The natural morphism  $\qlogo \rightarrow \cQ_{0,n}(X,\beta)$ given by forgetting the logarithmic map to $\agm$ factors through the inclusion $\cQ^{\mathrm{rel}}_{0,\alpha}(X | D, \beta) \hookrightarrow \cQ_{0,n}(X,\beta)$.
\end{proposition}

\begin{proof}
	Certainly on the locus where $C \cong \PP^1$ is irreducible and the section $u_0$ vanishes only at the marked points, this is true. Suppose now we have logarithmic quasimap which contains a connected component $Z \subset C$ on which $u_0 \equiv 0$. Then we need to show that $$\deg L_D|_Z + \sum\limits_{i} m^{(i)} = \sum\limits_{p_i \in Z} \alpha_i.$$
	Suppose first that $Z$ is irreducible. We have a logarithmic morphism $C \rightarrow \agm$. This induces a morphism of the tropicalisations. As in toric geometry, a piecewise linear function on the tropicalisation induces a Cartier divisor. Alternatively, a logarithmic structure can be characterised by an association of a Cartier divisor to each element of the ghost sheaf. The identity function on $\RR_{\geq 0}$ induces the divisor $\mathcal{B}\GG_m$. The fact that we have a logarithmic morphism necessarily implies that the pull back of this line bundle via the morphism is the same as the line bundle associated to the pull-back piecewise linear function. On the one hand the line bundle pulls back to $L_D$, when restricting to $Z$ we get $L_D|_Z$. On the other hand, the pull back piecewise linear function is defined by the slopes of the tropicalisation map on each ray. The associated line bundle on the component $Z$ is shown to be $\cO(\sum_{p_i \in Z} \alpha_i p_i - \sum_i m^{(i)}q_i)$, where $q_i$ are the nodes, in ~\cite[2.4.1]{ranganathan2019moduli}. Since these line bundles are necessarily isomorphic, taking degrees gives us the desired equality. If $Z$ is connected but reducible, then summing the resulting equalities obtained by the above over the irreducible components of $Z$ produces the result. 
\end{proof}

\begin{lemma}\label{PNHirred}
	Suppose $(X|D) = (\PP^N|H)$ where $H$ is the hyperplane defined (in coordinates) by $\{x_0 = 0\}$. Then $\cQ_{0,\alpha}^{\log}(\PP^N|H,d)$ is irreducible of the expected dimension $N\cdot (d+1) + n -3$.
\end{lemma}
\begin{proof}
	Recall the obstruction theory on $\cQ_{0,\alpha}^{\log}(\PP^N|H,d)$ is defined by the complex $\left(\mathbf{R} \pi_{*}\left(\mathbf{L}u^* \TT^{\log}_{[\AAA^{N+1}/\GG_m]}\right)\right)^{\vee}$
	Because $\PP^N|H$ admits a logarithmic Euler sequence \eqref{logeuler} $$0 \rightarrow \cO_{\PP^N} \rightarrow \cO_{\PP^N} \oplus \bigoplus_{i=1}^N \cO_{\PP^N}(1) \rightarrow \cT^{\log}_{\PP^N} \rightarrow 0.$$
	It follows from Lemma \ref{toricbdyvfc} that $\mathbf{L}u^*\TT^{\log}_{[\AAA^{N+1}/\GG_m]}$ is quasi-isomorphic to a sheaf $\cF$ which fits into an exact sequence on $\cC$ $$0 \rightarrow \cO_{\cC} \rightarrow \cO_{\cC} \oplus \bigoplus_{i=1}^N \cL \rightarrow \cF \rightarrow 0.$$
	Taking the long exact sequence in cohomology it follows that $$\mathrm{H}^1(C,\cF) = 0$$ and so this moduli space is unobstructed over $\frM_{0,\alpha}^{\log}(\agm)$, which is logarithmically smooth.
\end{proof}

To distinguish from a general $(X,D)$, we denote the morphism $\cQ_{0,\alpha}^{\mathrm{log}}(\PP^N|H,d) \rightarrow \cQ_{0,\alpha}^{\mathrm{rel}}(\PP^N|H,d)$ by $f$.
\begin{proposition}\label{BNPN}
	$$f_* [\cQ_{0,\alpha}^{\mathrm{log}}(\PP^N|H,d)]^{\vir} = [\cQ_{0,\alpha}^{\mathrm{rel}}(\PP^N|H,d)].$$
\end{proposition}
\begin{proof}
	By Lemma \ref{PNHirred} we have that $[\cQ_{0,\alpha}^{\mathrm{log}}(\PP^N|H,d)]^{\vir} = [\cQ_{0,\alpha}^{\mathrm{log}}(\PP^N|H,d)]$. So the proposition reduces to a statement about \emph{fundamental classes}. By Lemma \ref{PNHirred} and \cite{battistella2021relative} we know that the locus where the source curve is irreducible and the $u_0$ only vanishes at the marked points is dense in both spaces. Furthermore, on this locus the map is an isomorphism so the result follows.
\end{proof}

We now use Proposition \ref{BNPN} to prove Theorem \ref{eq}. 



\begin{proposition}\label{obstructionP}
	There is a perfect obstruction theory on $\mathcal{Q}^{\mathrm{log}}_{0,\alpha}(X|D,\beta)$ relative to $\mathcal{Q}^{\log}_{0,\alpha}(\PP^N|H,d)$ such that the corresponding virtual class, given by virtual pullback of $[\mathcal{Q}^{\log}_{0,\alpha}(\PP^N|H,d)]$, coincides with $[\cQ^{\log}_{0,\alpha}(X|D,\beta)]^{\vir}$.
\end{proposition}

\begin{proof}
	Recall $\cO_{X}(D)$ defined a morphism $j: X \hookrightarrow \PP^N$ such that $j^{-1}(H) = D$. If we write $X = \AAA^M \GIT \GG_m^s$ as a GIT quotient, as in Remark \ref{toricquotientstack}, then this morphism induces a morphism of quotient stacks, which we also denote $j$, $$j: [\AAA^M/\GG_m^s] \rightarrow [\AAA^{N+1}/\GG_m]$$ such that $j^{-1}(\bar{H}) = \bar{D}$, where $\bar{D}$ and $\bar{H}$ are the corresponding divisors on $[\AAA^M/\GG_m^s]$ and $[\AAA^{N+1}/\GG_m]$. The divisors $\bar{D},\bar{H}$ define logarithmic structures via morphisms to $[\AAA^1/\GG_m]$ such that there is a commutative diagram
	\begin{equation*}
		\begin{tikzcd}
		{[\AAA^M/\GG_m^s]} \arrow[r, "j"] \arrow[rr, "\bar{D}", bend left] & {[\AAA^{N+1}/\GG_m]} \arrow[r, "\bar{H}"] & \agm.
		\end{tikzcd}
	\end{equation*}

Taking the induced distinguished triangle (on tangent complexes) gives
    \begin{equation*}
\TT_{j} \rightarrow \TT^{\log}_{[\AAA^M/\GG_m^s]} \rightarrow \mathbf{L}j^*\TT^{\log}_{[\AAA^{N+1}/\GG_m]} \rightarrow \TT_{j}[1].
\end{equation*}

Next, consider the diagram
\begin{equation}\label{diagram0}
	\begin{tikzcd}
	\cC_{X} \arrow[d, "\pi"] \arrow[r] \arrow[rr, "u", bend left] & \cC_{\PP^N} \arrow[d, "\pi_{\PP}"] \arrow[rr, "u_{\PP}", bend right] & {[\AAA^M/\GG_m^s]} \arrow[r, "j"] & {[\AAA^{N+1}/\GG_m]} \\
	\qlogo \arrow[r, "i"]                                              & {\cQ^{\log}_{0,\alpha}(\PP^N|H,d)}.                                          &                                   &                     
	\end{tikzcd}
\end{equation}
If we apply $\mathbf{R}\pi_{*} \circ \mathbf{L}u^{*}$ and dualise, we get a distinguished triangle in $\qlogo$
\begin{equation*}
\mathbf{L}i^*(\mathbf{R}(\pi_{\PP})_{*}\mathbf{L}u_{\PP}^*\TT^{\log}_{[\AAA^{N+1}/\GG_m]})^{\vee} \rightarrow 	(\mathbf{R}\pi_{*}\mathbf{L}u^*\TT^{\log}_{[\AAA^M/\GG_m^s]})^{\vee} \rightarrow (\mathbf{R}\pi_{*}\mathbf{L}u^*\TT_{j})^{\vee} \rightarrow \mathbf{L}i^*(\mathbf{R}(\pi_{\PP})_{*}\mathbf{L}u_{\PP}^*\TT^{\log}_{[\AAA^{N+1}/\GG_m]})^{\vee}[1].
\end{equation*}
Note that the first two complexes in the triangle define the obstruction theories for $\qlogo$ and $\cQ^{\log}_{0,\alpha}(\PP^N|H,d)$.

We also have the commutative diagram
	\begin{equation*}
\begin{tikzcd}
{\qlogo} \arrow[r, "i"] \arrow[rr, bend left] & {\cQ^{\log}_{0,\alpha}(\PP^N|H,d)} \arrow[r] & \frM^{\log}_{0,\alpha}(\agm)
\end{tikzcd}
\end{equation*}
inducing a distinguished triangle $$\mathbf{L}i^*\LL_{\cQ^{\log}(\PP)/\frM^{\log}} \rightarrow \LL_{\cQ^{\log}(X)/\frM^{\log}} \rightarrow \LL_{i} \rightarrow \mathbf{L}i^*\LL_{\cQ^{\log}(\PP)/\frM^{\log}}[1]$$

where $\cQ^{\log}(X) = \cQ^{\log}_{0,\alpha}(X|D,\beta), \, \cQ^{\log}(\PP) = \cQ^{\log}_{0,\alpha}(\PP^N|H,d)$ and $\frM^{\log} = \frM^{\log}_{0,\alpha}(\agm)$. Putting these together we get 
\begin{equation*}
	\begin{tikzcd}
{\mathbf{L}i^*(\mathbf{R}(\pi_{\PP})_{*}\mathbf{L}u_{\PP}^*\TT^{\log}_{[\AAA^{N+1}/\GG_m]})^{\vee}} \arrow[r] \arrow[d] & {(\mathbf{R}\pi_{*}\mathbf{L}u^*\TT^{\log}_{[\AAA^M/\GG_m^s]})^{\vee}} \arrow[r] \arrow[d] & (\mathbf{R}\pi_{*}\mathbf{L}u^*\TT_{j})^{\vee} \arrow[r, "{[1]}"] \arrow[d] & {} \\
\mathbf{L}i^*\LL_{\cQ^{\log}(\PP)/\frM^{\log}} \arrow[r]                                                                & \LL_{\cQ^{\log}(X)/\frM^{\log}} \arrow[r]                                                  & \LL_{i} \arrow[r, "{[1]}"]                                                  & {}.
\end{tikzcd}
\end{equation*}

The first two vertical arrows come from the perfect obstruction theory from Theorem \ref{logqvir}. By ~\cite[Construction 3.13]{manolache12virtual} and ~\cite[Remark 3.15]{manolache12virtual} it follows that $(\mathbf{R}\pi_{*}\mathbf{L}u^*\TT_{j})^{\vee}$ defines a perfect obstruction theory. Moreover in the language of ~\cite[Definition 4.5]{manolache12virtual} the three perfect obstruction theories form a compatible triple and so by ~\cite[Theorem 4.8]{manolache12virtual} we have that $i^{!}[\cQ_{0,\alpha}^{\log}(\PP^N|H,d)] = [\qlogo]^{\vir}$. 

\end{proof}

\begin{proof}[Proof of Theorem \ref{eq}]
	Consider the tower of cartesian diagrams
		\begin{equation*}
\begin{tikzcd}
{\mathcal{Q}^{\mathrm{\log}}_{0,\alpha}(X|D,\beta)} \arrow[d, "g"] \arrow[r, "i"] & {\mathcal{Q}^{\mathrm{\log}}_{0,\alpha}(\mathbb{P}^N|H,d)} \arrow[d, "f"] \\
{\mathcal{Q}^{\mathrm{rel}}_{0,\alpha}(X|D,\beta)} \arrow[r, "i'"] \arrow[d]      & {\mathcal{Q}^{\mathrm{rel}}_{0,\alpha}(\mathbb{P}^N|H,d)} \arrow[d]   \\
{\cQ_{0,n}(X,\beta)} \arrow[r, "j'"]                                              & {\cQ_{0,n}(\PP^N,d)}                                                     
\end{tikzcd}
	\end{equation*}
	
The fact that the top square is cartesian follows from the fact that the both the bottom and large squares are cartesian. We want to show that $g_{\ast}[\qlogo]^{\mathrm{vir}} = [\qrel]^{\vir}$. By Proposition \ref{BNPN} we have that $f_{*}[\cQ_{0,\alpha}^{\log}(\PP^N|H,d)] = [\cQ_{0,\alpha}^{\mathrm{rel}}(\PP^N|H,d)]$. In Proposition \ref{obstructionP} we showed that $[\cQ^{\log}_{0,\alpha}(X|D,\beta)]^{\vir}$ is defined via a perfect obstruction theory for $i$. In actual fact this obstruction theory is pulled back from $j'$. To see this, repeat the argument of Proposition \ref{obstructionP}, starting instead with the distinguished triangle $$\TT_{j} \rightarrow \TT_{[\AAA^M/\GG_m^s]} \rightarrow \mathbf{L}j^*\TT_{[\AAA^{N+1}/\GG_m]} \rightarrow \TT_{j}[1]$$ which shows that the perfect obstruction theory for $j'$ also pulls back to $(\mathbf{R}\pi_{*}\mathbf{L}u^*\TT_{j})^{\vee}$. This tells us that ${j'}^![\cQ_{0,\alpha}^{\log}(\PP^N|H,d)] = [\cQ_{0,\alpha}^{\log}(X|D,\beta)]^{\vir}$ and since diagonal pullback coincides with virtual pullback ~\cite[Lemma A.0.1]{battistella2021relative}, we have that ${j'}^![\cQ_{0,\alpha}^{\mathrm{rel}}(\PP^N|H,d)] = [\qrel]^{\vir}$. Therefore, the theorem follows from ~\cite[Proposition 5.29]{manolache12virtual}.

\end{proof}

\bibliography{bibliographyLOG}

\begin{thebibliography}{10}

\bibitem{abramovich2017relative}
D.~Abramovich, C.~Cadman, and J.~Wise.
\newblock Relative and orbifold {G}romov-{W}itten invariants.
\newblock {\em Algebr. Geom.}, 4(4):472--500, 2017.

\bibitem{abramovich2014stable}
D.~Abramovich and Q.~Chen.
\newblock Stable logarithmic maps to {D}eligne-{F}altings pairs {II}.
\newblock {\em Asian J. Math.}, 18(3):465--488, 2014.

\bibitem{abramovich2013logarithmic}
D.~Abramovich, Q.~Chen, D.~Gillam, Y.~Huang, M.~Olsson, M.~Satriano, and S.~Sun.
\newblock Logarithmic geometry and moduli.
\newblock In {\em Handbook of moduli. {V}ol. {I}}, volume~24 of {\em Adv. Lect. Math. (ALM)}, pages 1--61. Int. Press, Somerville, MA, 2013.

\bibitem{abramovich2020decomposition}
D.~Abramovich, Q.~Chen, M.~Gross, and B.~Siebert.
\newblock Decomposition of degenerate {G}romov-{W}itten invariants.
\newblock {\em Compos. Math.}, 156(10):2020--2075, 2020.

\bibitem{abramovich2021punctured}
D.~Abramovich, Q.~Chen, M.~Gross, and B.~Siebert.
\newblock Punctured logarithmic maps, 2021.

\bibitem{abramovich2016orbifold}
D.~Abramovich and B.~Fantechi.
\newblock Orbifold techniques in degeneration formulas.
\newblock {\em Ann. Sc. Norm. Super. Pisa Cl. Sci. (5)}, 16(2):519--579, 2016.

\bibitem{abramovich2014comparison}
D.~Abramovich, S.~Marcus, and J.~Wise.
\newblock Comparison theorems for {G}romov-{W}itten invariants of smooth pairs and of degenerations.
\newblock {\em Ann. Inst. Fourier (Grenoble)}, 64(4):1611--1667, 2014.

\bibitem{abramovich2018birational}
D.~Abramovich and J.~Wise.
\newblock Birational invariance in logarithmic {G}romov-{W}itten theory.
\newblock {\em Compos. Math.}, 154(3):595--620, 2018.

\bibitem{battistella2021relative}
L.~Battistella and N.~Nabijou.
\newblock Relative quasimaps and mirror formulae.
\newblock {\em Int. Math. Res. Not. IMRN}, (10):7885--7931, 2021.

\bibitem{battistella2022gromov}
L.~Battistella, N.~Nabijou, and D.~Ranganathan.
\newblock Gromov-{W}itten theory via roots and logarithms, 2022.

\bibitem{battistella2022local}
L.~Battistella, N.~Nabijou, H.-H. Tseng, and F.~You.
\newblock The local-orbifold correspondence for simple normal crossing pairs.
\newblock {\em Journal of the Institute of Mathematics of Jussieu}, page 1–17, 2022.

\bibitem{behrend1997intrinsic}
K.~Behrend and B.~Fantechi.
\newblock The intrinsic normal cone.
\newblock {\em Invent. Math.}, 128(1):45--88, 1997.

\bibitem{behrend2005derham}
K.~A. Behrend.
\newblock On the de {R}ham cohomology of differential and algebraic stacks.
\newblock {\em Adv. Math.}, 198(2):583--622, 2005.

\bibitem{bousseau2020quantum}
P.~Bousseau.
\newblock The quantum tropical vertex.
\newblock {\em Geom. Topol.}, 24(3):1297--1379, 2020.

\bibitem{bousseau2021holomorphic}
P.~Bousseau, H.~Fan, S.~Guo, and L.~Wu.
\newblock Holomorphic anomaly equation for {$(\mathbb{P}^2,E)$} and the {N}ekrasov-{S}hatashvili limit of local {$\mathbb{P}^2$}.
\newblock {\em Forum Math. Pi}, 9:Paper No. e3, 57, 2021.

\bibitem{cadman2007using}
C.~Cadman.
\newblock Using stacks to impose tangency conditions on curves.
\newblock {\em Amer. J. Math.}, 129(2):405--427, 2007.

\bibitem{chen2014degeneration}
Q.~Chen.
\newblock The degeneration formula for logarithmic expanded degenerations.
\newblock {\em J. Algebraic Geom.}, 23(2):341--392, 2014.

\bibitem{chen2014stable}
Q.~Chen.
\newblock Stable logarithmic maps to {D}eligne-{F}altings pairs {I}.
\newblock {\em Ann. of Math. (2)}, 180(2):455--521, 2014.

\bibitem{cheong2015orbifold}
D.~Cheong, I.~Ciocan-Fontanine, and B.~Kim.
\newblock Orbifold quasimap theory.
\newblock {\em Math. Ann.}, 363(3-4):777--816, 2015.

\bibitem{ciocan2010moduli}
I.~Ciocan-Fontanine and B.~Kim.
\newblock Moduli stacks of stable toric quasimaps.
\newblock {\em Adv. Math.}, 225(6):3022--3051, 2010.

\bibitem{ciocan2014wall}
I.~Ciocan-Fontanine and B.~Kim.
\newblock Wall-crossing in genus zero quasimap theory and mirror maps.
\newblock {\em Algebr. Geom.}, 1(4):400--448, 2014.

\bibitem{ciocan2020quasimap}
I.~Ciocan-Fontanine and B.~Kim.
\newblock Quasimap wall-crossings and mirror symmetry.
\newblock {\em Publ. Math. Inst. Hautes \'{E}tudes Sci.}, 131:201--260, 2020.

\bibitem{ciocan2014stable}
I.~Ciocan-Fontanine, B.~Kim, and D.~Maulik.
\newblock Stable quasimaps to {GIT} quotients.
\newblock {\em J. Geom. Phys.}, 75:17--47, 2014.

\bibitem{coates2022abelian}
T.~Coates, W.~Lutz, and Q.~Shafi.
\newblock The abelian/nonabelian correspondence and {G}romov-{W}itten invariants of blow-ups.
\newblock {\em Forum Math. Sigma}, 10:Paper No. e67, 33, 2022.

\bibitem{cox1995homogeneous}
D.~A. Cox.
\newblock The homogeneous coordinate ring of a toric variety.
\newblock {\em J. Algebraic Geom.}, 4(1):17--50, 1995.

\bibitem{cox2011toric}
D.~A. Cox, J.~B. Little, and H.~K. Schenck.
\newblock {\em Toric varieties}, volume 124 of {\em Graduate Studies in Mathematics}.
\newblock American Mathematical Society, Providence, RI, 2011.

\bibitem{fan2019mirror}
H.~Fan, H.-H. Tseng, and F.~You.
\newblock Mirror theorems for root stacks and relative pairs.
\newblock {\em Selecta Math. (N.S.)}, 25(4):Paper No. 54, 25, 2019.

\bibitem{fan2020structures}
H.~Fan, L.~Wu, and F.~You.
\newblock Structures in genus-zero relative {G}romov-{W}itten theory.
\newblock {\em J. Topol.}, 13(1):269--307, 2020.

\bibitem{fan2021higher}
H.~Fan, L.~Wu, and F.~You.
\newblock Higher genus relative {G}romov-{W}itten theory and double ramification cycles.
\newblock {\em J. Lond. Math. Soc. (2)}, 103(4):1547--1576, 2021.

\bibitem{fausk2003isomorphisms}
H.~Fausk, P.~Hu, and J.~P. May.
\newblock Isomorphisms between left and right adjoints.
\newblock {\em Theory Appl. Categ.}, 11:No. 4, 107--131, 2003.

\bibitem{gathmann2002absolute}
A.~Gathmann.
\newblock Absolute and relative {G}romov-{W}itten invariants of very ample hypersurfaces.
\newblock {\em Duke Math. J.}, 115(2):171--203, 2002.

\bibitem{gillam2012logarithmic}
W.~D. Gillam.
\newblock Logarithmic stacks and minimality.
\newblock {\em Internat. J. Math.}, 23(7):1250069, 38, 2012.

\bibitem{graber2005relative}
T.~Graber and R.~Vakil.
\newblock Relative virtual localization and vanishing of tautological classes on moduli spaces of curves.
\newblock {\em Duke Math. J.}, 130(1):1--37, 2005.

\bibitem{graefnitz2022tropical}
T.~Graefnitz.
\newblock Tropical correspondence for smooth del {P}ezzo log {C}alabi-{Y}au pairs.
\newblock {\em J. Algebraic Geom.}, 31(4):687--749, 2022.

\bibitem{gross2011book}
M.~Gross.
\newblock {\em Tropical geometry and mirror symmetry}, volume 114 of {\em CBMS Regional Conference Series in Mathematics}.
\newblock Published for the Conference Board of the Mathematical Sciences, Washington, DC; by the American Mathematical Society, Providence, RI, 2011.

\bibitem{gross2010tropical}
M.~Gross, R.~Pandharipande, and Bernd Siebert.
\newblock The tropical vertex.
\newblock {\em Duke Math. J.}, 153(2):297--362, 2010.

\bibitem{gross2013logarithmic}
M.~Gross and B.~Siebert.
\newblock Logarithmic {G}romov-{W}itten invariants.
\newblock {\em J. Amer. Math. Soc.}, 26(2):451--510, 2013.

\bibitem{gross2022canonical}
M.~Gross and B.~Siebert.
\newblock The canonical wall structure and intrinsic mirror symmetry.
\newblock {\em Invent. Math.}, 229(3):1101--1202, 2022.

\bibitem{illusie71complexe}
L.~Illusie.
\newblock {\em Complexe cotangent et d\'{e}formations. {I}}.
\newblock Lecture Notes in Mathematics, Vol. 239. Springer-Verlag, Berlin-New York, 1971.

\bibitem{illusie2005grothendieck}
L.~Illusie.
\newblock Grothendieck's existence theorem in formal geometry.
\newblock In {\em Fundamental algebraic geometry}, volume 123 of {\em Math. Surveys Monogr.}, pages 179--233. Amer. Math. Soc., Providence, RI, 2005.
\newblock With a letter (in French) of Jean-Pierre Serre.

\bibitem{ionel2003relative}
E.-N. Ionel and T.~H. Parker.
\newblock Relative {G}romov-{W}itten invariants.
\newblock {\em Ann. of Math. (2)}, 157(1):45--96, 2003.

\bibitem{kennedy2022divisors}
P.~Kennedy-Hunt, N.~Nabijou, Q.~Shafi, and W.~Zheng.
\newblock Divisors and curves on logarithmic mapping spaces, 2022.

\bibitem{kim2010logarithmic}
B.~Kim.
\newblock Logarithmic stable maps.
\newblock In {\em New developments in algebraic geometry, integrable systems and mirror symmetry ({RIMS}, {K}yoto, 2008)}, volume~59 of {\em Adv. Stud. Pure Math.}, pages 167--200. Math. Soc. Japan, Tokyo, 2010.

\bibitem{kim2021degeneration}
B.~Kim, H.~Lho, and H.~Ruddat.
\newblock The degeneration formula for stable log maps.
\newblock {\em Manuscripta Math.}, 2021.

\bibitem{gerard00champs}
G.~Laumon and L.~Moret-Bailly.
\newblock {\em Champs alg\'{e}briques}, volume~39 of {\em Ergebnisse der Mathematik und ihrer Grenzgebiete. 3. Folge. A Series of Modern Surveys in Mathematics [Results in Mathematics and Related Areas. 3rd Series. A Series of Modern Surveys in Mathematics]}.
\newblock Springer-Verlag, Berlin, 2000.

\bibitem{lho2018stable}
H.~Lho and R.~Pandharipande.
\newblock Stable quotients and the holomorphic anomaly equation.
\newblock {\em Adv. Math.}, 332:349--402, 2018.

\bibitem{li2001symplectic}
A.-M Li and Y.~Ruan.
\newblock Symplectic surgery and {G}romov-{W}itten invariants of {C}alabi-{Y}au 3-folds.
\newblock {\em Invent. Math.}, 145(1):151--218, 2001.

\bibitem{li2001stable}
J.~Li.
\newblock Stable morphisms to singular schemes and relative stable morphisms.
\newblock {\em J. Differential Geom.}, 57(3):509--578, 2001.

\bibitem{li2002degeneration}
J.~Li.
\newblock A degeneration formula of {GW}-invariants.
\newblock {\em J. Differential Geom.}, 60(2):199--293, 2002.

\bibitem{mandel2020descendant}
T.~Mandel and H.~Ruddat.
\newblock Descendant log {G}romov-{W}itten invariants for toric varieties and tropical curves.
\newblock {\em Trans. Amer. Math. Soc.}, 373(2):1109--1152, 2020.

\bibitem{manolache12virtual}
C.~Manolache.
\newblock Virtual pull-backs.
\newblock {\em J. Algebraic Geom.}, 21(2):201--245, 2012.

\bibitem{marian2011moduli}
A.~Marian, D.~Oprea, and R.~Pandharipande.
\newblock The moduli space of stable quotients.
\newblock {\em Geom. Topol.}, 15(3):1651--1706, 2011.

\bibitem{nabijou2019recursion}
N.~Nabijou.
\newblock {\em Recursion Formulae in Logarithmic Gromov-Witten Theory and Quasimap Theory}.
\newblock PhD thesis, Imperial College London, 2019.

\bibitem{nabijou2022gromov}
N.~Nabijou and D.~Ranganathan.
\newblock Gromov-{W}itten theory with maximal contacts.
\newblock {\em Forum Math. Sigma}, 10:Paper No. e5, 34, 2022.

\bibitem{nishinou2006toric}
T.~Nishinou and B.~Siebert.
\newblock Toric degenerations of toric varieties and tropical curves.
\newblock {\em Duke Math. J.}, 135(1):1--51, 2006.

\bibitem{olsson2003logarithmic}
M.~C. Olsson.
\newblock Logarithmic geometry and algebraic stacks.
\newblock {\em Ann. Sci. \'{E}cole Norm. Sup. (4)}, 36(5):747--791, 2003.

\bibitem{olsson2005logarithmic}
M.~C. Olsson.
\newblock The logarithmic cotangent complex.
\newblock {\em Math. Ann.}, 333(4):859--931, 2005.

\bibitem{olsson2006deformation}
M.~C. Olsson.
\newblock Deformation theory of representable morphisms of algebraic stacks.
\newblock {\em Math. Z.}, 253(1):25--62, 2006.

\bibitem{pandharipande2012kappa}
R.~Pandharipande.
\newblock {The $\kappa$ ring of the moduli of curves of compact type}.
\newblock {\em Acta Mathematica}, 208(2):335 -- 388, 2012.

\bibitem{ranganathan2017skeletons}
D.~Ranganathan.
\newblock Skeletons of stable maps {I}: rational curves in toric varieties.
\newblock {\em J. Lond. Math. Soc. (2)}, 95(3):804--832, 2017.

\bibitem{ranganathan2019logarithmic}
D.~Ranganathan.
\newblock Logarithmic {G}romov-{W}itten theory with expansions.
\newblock {\em arXiv preprint arXiv:1903.09006}, 2019.

\bibitem{ranganathan2019moduli}
D.~Ranganathan, K.~Santos-Parker, and J.~Wise.
\newblock Moduli of stable maps in genus one and logarithmic geometry, {I}.
\newblock {\em Geom. Topol.}, 23(7):3315--3366, 2019.

\bibitem{shafi2022enumerative}
Q.~Shafi.
\newblock {\em Enumerative Geometry of GIT Quotients}.
\newblock PhD thesis, Imperial College London, 2022.

\bibitem{tseng2020gromovwitten}
H.-H. Tseng and F.~You.
\newblock A {G}romov-{W}itten theory for simple normal-crossing pairs without log geometry, 2020.

\bibitem{vakil2000enumerative}
R.~Vakil.
\newblock The enumerative geometry of rational and elliptic curves in projective space.
\newblock {\em J. Reine Angew. Math.}, 529:101--153, 2000.

\bibitem{vangarrel2019local}
M.~van Garrel, T.~Graber, and H.~Ruddat.
\newblock Local {G}romov-{W}itten invariants are log invariants.
\newblock {\em Adv. Math.}, 350:860--876, 2019.

\bibitem{you2022relative}
F.~You.
\newblock Relative quantum cohomology under birational transformations, 2022.

\bibitem{zhou2022quasimap}
Y.~Zhou.
\newblock Quasimap wall-crossing for {GIT} quotients.
\newblock {\em Invent. Math.}, 227(2):581--660, 2022.

\end{thebibliography}
\bibliographystyle{plain}

\footnotesize

\end{document}